\documentclass[11pt,letterpaper]{article}

\usepackage{fullpage}

\usepackage{graphicx,subfigure}

\def\showauthornotes{2}

\def\showkeys{0}
\def\showdraftbox{1}
\def\showcolorlinks{1}
\def\usemicrotype{1}
\def\showfixme{1}

\def\arxivmode{0}
\def\fastmode{0}

\newcommand{\llangle}{\left\langle}
\newcommand{\rrangle}{\right\rangle}

\ifnum\fastmode=0
\usepackage{etex}
\fi

\ifnum\fastmode=0
\usepackage[l2tabu, orthodox]{nag}
\fi

\usepackage{xspace,enumerate}

\usepackage[dvipsnames]{xcolor}

\ifnum\fastmode=0
\usepackage[T1]{fontenc}
\usepackage[full]{textcomp}
\fi

\usepackage[american]{babel}

\usepackage{mathtools}

\usepackage{amsthm}

\newtheorem{theorem}{Theorem}[section]
\newtheorem*{theorem*}{Theorem}

\newtheorem*{proposition*}{Proposition}
\newtheorem{lemma}[theorem]{Lemma}
\newtheorem*{lemma*}{Lemma}
\newtheorem{corollary}[theorem]{Corollary}
\newtheorem*{conjecture*}{Conjecture}

\newtheorem*{fact*}{Fact}

\newtheorem*{exercise*}{Exercise}

\newtheorem*{hypothesis*}{Hypothesis}

\theoremstyle{definition}
\newtheorem{definition}[theorem]{Definition}

\newtheorem{exercise-easy}[theorem]{Exercise}
\newtheorem{exercise-med}[theorem]{Exercise}
\newtheorem{exercise-hard}[theorem]{Exercise$^\star$}
\newtheorem{claim}[theorem]{Claim}
\newtheorem*{claim*}{Claim}

\newtheorem{remark}[theorem]{Remark}
\newtheorem*{remark*}{Remark}
\newtheorem{observation}[theorem]{Observation}
\newtheorem*{observation*}{Observation}

\ifnum\arxivmode=0
\usepackage{newpxtext} %
\usepackage{textcomp} %
\usepackage[varg,bigdelims]{newpxmath}
\usepackage[scr=rsfso]{mathalfa}%
\usepackage{bm} %
\let\mathbb\varmathbb
\fi

\ifnum\arxivmode=1
\usepackage[varg]{pxfonts} %
\usepackage{textcomp} %
\usepackage[scr=rsfso]{mathalfa}%
\usepackage{bm} %
\fi

\ifnum\showkeys=1
\usepackage[color]{showkeys}
\fi

\makeatletter
\@ifpackageloaded{hyperref}
{}
{\usepackage[pagebackref]{hyperref}}

\definecolor{bleudefrance}{rgb}{0.01, 0.1, 1.0}
\definecolor{azure}{rgb}{0.0, 0.5, 1.0}

\ifnum\showcolorlinks=1
\hypersetup{
pagebackref=true,
colorlinks=true,
urlcolor=blue,
linkcolor=blue,
citecolor=OliveGreen}
\fi

\ifnum\showcolorlinks=0
\hypersetup{
pagebackref=true,
colorlinks=false,
pdfborder={0 0 0}}
\fi

\usepackage{prettyref}

\newcommand{\savehyperref}[2]{\texorpdfstring{\hyperref[#1]{#2}}{#2}}

\newrefformat{eq}{\savehyperref{#1}{\textup{(\ref*{#1})}}}
\newrefformat{lem}{\savehyperref{#1}{Lemma~\ref*{#1}}}
\newrefformat{def}{\savehyperref{#1}{Definition~\ref*{#1}}}
\newrefformat{obs}{\savehyperref{#1}{Observation~\ref*{#1}}}
\newrefformat{thm}{\savehyperref{#1}{Theorem~\ref*{#1}}}
\newrefformat{cor}{\savehyperref{#1}{Corollary~\ref*{#1}}}
\newrefformat{cha}{\savehyperref{#1}{Chapter~\ref*{#1}}}
\newrefformat{sec}{\savehyperref{#1}{Section~\ref*{#1}}}
\newrefformat{app}{\savehyperref{#1}{Appendix~\ref*{#1}}}
\newrefformat{tab}{\savehyperref{#1}{Table~\ref*{#1}}}
\newrefformat{fig}{\savehyperref{#1}{Figure~\ref*{#1}}}
\newrefformat{hyp}{\savehyperref{#1}{Hypothesis~\ref*{#1}}}
\newrefformat{alg}{\savehyperref{#1}{Algorithm~\ref*{#1}}}
\newrefformat{rem}{\savehyperref{#1}{Remark~\ref*{#1}}}
\newrefformat{item}{\savehyperref{#1}{Item~\ref*{#1}}}
\newrefformat{step}{\savehyperref{#1}{step~\ref*{#1}}}
\newrefformat{conj}{\savehyperref{#1}{Conjecture~\ref*{#1}}}
\newrefformat{fact}{\savehyperref{#1}{Fact~\ref*{#1}}}
\newrefformat{prop}{\savehyperref{#1}{Proposition~\ref*{#1}}}
\newrefformat{prob}{\savehyperref{#1}{Problem~\ref*{#1}}}
\newrefformat{claim}{\savehyperref{#1}{Claim~\ref*{#1}}}
\newrefformat{relax}{\savehyperref{#1}{Relaxation~\ref*{#1}}}
\newrefformat{red}{\savehyperref{#1}{Reduction~\ref*{#1}}}
\newrefformat{part}{\savehyperref{#1}{Part~\ref*{#1}}}
\newrefformat{ex}{\savehyperref{#1}{Exercise~\ref*{#1}}}
\newrefformat{property}{\savehyperref{#1}{Property~\ref*{#1}}}

\newcommand{\Sref}[1]{\hyperref[#1]{\S\ref*{#1}}}

\usepackage{nicefrac}

\ifnum\fastmode=0
\ifnum\usemicrotype=1
\usepackage{microtype}
\fi
\fi

\ifnum\showauthornotes=2
\newcommand{\mynotes}[1]{{\sffamily\small\color{teal}{#1}}\medskip}
\newcommand{\Authornote}[2]{{\sffamily\small\color{Maroon}{[#1: #2]}}\medskip}
\newcommand{\Authornotecolored}[3]{{\sffamily\small\color{#1}{[#2: #3]}}}
\newcommand{\Authorcomment}[2]{{\sffamily\small\color{gray}{[#1: #2]}}}
\newcommand{\Authorstartcomment}[1]{\sffamily\small\color{gray}[#1: }

\newcommand{\Authorfnote}[2]{\footnote{\color{red}{#1: #2}}}
\newcommand{\Authorfixme}[1]{\Authornote{#1}{\textbf{??}}}
\newcommand{\Authormarginmark}[1]{\marginpar{\textcolor{red}{\fbox{\Large #1:!}}}}
\newcommand{\myexplain}[1]{{\sffamily\small\color{red}{\noindent [Explanation:\medskip\newline \begin{quote}#1\hfill]\end{quote}}}\medskip}
\newcommand{\explain}[1]{{\sffamily\small\color{red}{#1}}\medskip}

\else

\newcommand{\mynotes}[1]{}
\newcommand{\Authornote}[2]{}
\newcommand{\Authornotecolored}[3]{}
\newcommand{\Authorcomment}[2]{}
\newcommand{\Authorstartcomment}[1]{}

\newcommand{\Authorfnote}[2]{}
\newcommand{\Authorfixme}[1]{}
\newcommand{\Authormarginmark}[1]{}
\newcommand{\myexplain}[1]{}
\newcommand{\explain}[1]{}

\ifnum\showauthornotes=1
\renewcommand{\myexplain}[1]{{\sffamily\small\color{red}{\noindent \begin{quote}{\bf Explanation:} \medskip\newline #1\end{quote}}}\medskip}
\fi

\fi

\ifnum\showfixme=0

\fi

\usepackage{boxedminipage}

\newcommand{\Esymb}{\mathbb{E}}
\newcommand{\Psymb}{\mathbb{P}}

\DeclareMathOperator*{\E}{\Esymb}

\DeclareMathOperator*{\ProbOp}{\Psymb}

\renewcommand{\Pr}{\ProbOp}

\newcommand{\textparen}[1]{\text{(#1)}}

\ifx\because\undefined
\newcommand{\because}[1]{\textparen{because #1}}
\else
\renewcommand{\because}[1]{\textparen{because #1}}
\fi

\newcommand{\seteq}{\mathrel{\mathop:}=}

\newcommand\bdot\bullet

\ifx\mathds\undefined %

\else

\fi

\DeclareMathOperator{\supp}{supp}

\newcommand{\N}{\mathbb N}
\newcommand{\R}{\mathbb R}

\newcommand{\cC}{\mathcal C}
\newcommand{\cD}{\mathcal D}
\newcommand{\cE}{\mathcal E}
\newcommand{\cF}{\mathcal F}
\newcommand{\cG}{\mathcal G}
\newcommand{\cH}{\mathcal H}

\newcommand{\cL}{\mathcal L}

\newcommand{\cR}{\mathcal R}

\newcommand{\cT}{\mathcal T}

\newcommand{\bA}{\bm A}

\newcommand{\bH}{\bm H}
\newcommand{\bI}{\bm I}

\newcommand{\bS}{\bm S}
\newcommand{\bT}{\bm T}
\newcommand{\bU}{\bm U}

\newcommand{\scrL}{\mathscr L}

\newcommand{\scrR}{\mathscr R}

\newcommand{\bbS}{\mathbb S}

\renewcommand{\leq}{\leqslant}

\renewcommand{\geq}{\geqslant}

\ifnum\showdraftbox=1

\else

\fi

\let\epsilon=\varepsilon

\numberwithin{equation}{section}

\newcommand\MYcurrentlabel{xxx}

\newcommand{\MYstore}[2]{%
  \global\expandafter \def \csname MYMEMORY #1 \endcsname{#2}%
}

\newcommand{\MYload}[1]{%
  \csname MYMEMORY #1 \endcsname%
}

\newcommand{\MYnewlabel}[1]{%
  \renewcommand\MYcurrentlabel{#1}%
  \MYoldlabel{#1}%
}

\newcommand{\MYdummylabel}[1]{}

\newcommand{\torestate}[1]{%
  \let\MYoldlabel\label%
  \let\label\MYnewlabel%
  #1%
  \MYstore{\MYcurrentlabel}{#1}%
  \let\label\MYoldlabel%
}

\newcommand{\restatetheorem}[1]{%
  \let\MYoldlabel\label
  \let\label\MYdummylabel
  \begin{theorem*}[Restatement of \prettyref{#1}]
    \MYload{#1}
  \end{theorem*}
  \let\label\MYoldlabel
}

\newcommand{\restatelemma}[1]{%
  \let\MYoldlabel\label
  \let\label\MYdummylabel
  \begin{lemma*}[Restatement of \prettyref{#1}]
    \MYload{#1}
  \end{lemma*}
  \let\label\MYoldlabel
}

\newcommand{\restateprop}[1]{%
  \let\MYoldlabel\label
  \let\label\MYdummylabel
  \begin{proposition*}[Restatement of \prettyref{#1}]
    \MYload{#1}
  \end{proposition*}
  \let\label\MYoldlabel
}

\newcommand{\restatefact}[1]{%
  \let\MYoldlabel\label
  \let\label\MYdummylabel
  \begin{fact*}[Restatement of \prettyref{#1}]
    \MYload{#1}
  \end{fact*}
  \let\label\MYoldlabel
}

\newcommand{\restate}[1]{%
  \let\MYoldlabel\label
  \let\label\MYdummylabel
  \MYload{#1}
  \let\label\MYoldlabel
}

\newcommand{\addreferencesection}{
  \phantomsection
\ifnum\stocmode=0
  \addcontentsline{toc}{section}{References}
\else
  \addcontentsline{toc}{section}{References \hspace*{1in} --------- End of extended abstract ---------}
\fi

}

\newcommand{\e}{\epsilon}
\newcommand{\eps}{\epsilon}

\let\origparagraph\paragraph
\renewcommand{\paragraph}[1]{\vspace*{-3pt}\origparagraph{#1.}}

\allowdisplaybreaks

\usepackage{paralist}

\usepackage{comment}

\let\pref=\prettyref

\newcommand{\diam}{\mathrm{diam}}

\newcommand{\vertiii}[1]{{\left\vert\kern-0.25ex\left\vert\kern-0.25ex\left\vert #1 
          \right\vert\kern-0.25ex\right\vert\kern-0.25ex\right\vert}}

\usepackage{enumitem}
\usepackage[normalem]{ulem}
\usepackage{stmaryrd}

\newcommand{\cmnt}[1]{}
\newcommand{\reff}{\mathsf{R}_{\mathrm{eff}}}

\newcommand{\region}[1]{\left\llbracket #1\right\rrbracket}
\newcommand{\1}{\mathbb{1}}
\newcommand{\todl}{\Rightarrow}
\renewcommand{\mathbb}{\vvmathbb}
\newcommand{\dloc}{\vvmathbb{d}_{\mathrm{loc}}}
\newcommand{\graphs}{\cG}
\newcommand{\rgraphs}{\cG_{\bullet}}

\newcommand{\lra}{\leftrightarrow}

\newcommand{\tileprod}{\circ}
\newcommand{\tilecat}{\mid}

\title{On planar graphs of uniform polynomial growth}
\author{Farzam Ebrahimnejad\thanks{\texttt{febrahim@cs.washington.edu}}\hspace{0.3in}  James R. Lee\thanks{\texttt{jrl@cs.washington.edu}} 
   \\ {\small Paul G. Allen School of Computer Science \& Engineering} \\ {\small University of Washington}}

\date{}

\begin{document}

\maketitle

\begin{abstract}
   Consider an infinite planar graph with uniform polynomial growth
   of degree $d > 2$.  Many examples of such graphs exhibit
   similar geometric and spectral properties, and it has been conjectured that this is necessary.
   We present a family of counterexamples.
   In particular, we show that for every rational $d > 2$, there is a planar graph with uniform polynomial
   growth of degree $d$
   on which the random walk is transient, disproving a conjecture of Benjamini (2011).

   By a well-known theorem of Benjamini and Schramm, such a graph cannot be a unimodular random graph.
   We also give examples of unimodular random planar graphs of uniform polynomial growth
   with unexpected properties.  For instance, graphs of (almost sure) uniform polynomial growth of every rational
   degree $d > 2$ for which the speed exponent of the walk is larger than $1/d$,
   and in which the complements of all balls are connected.
   This resolves negatively two questions of Benjamini and Papasoglou (2011).
\end{abstract}

\begingroup
\hypersetup{linktocpage=false}
\setcounter{tocdepth}{2}
\tableofcontents
\endgroup

\newpage

\section{Introduction}

Say that a graph $G$ has {\em uniform polynomial growth of degree $d$} if the cardinality
of all balls of radius $r$ in the graph metric lie between $c r^d$ and $C r^d$ for
two absolute constants $C > c > 0$, for every $r > 0$.
Say that a graph has {\em nearly-uniform polynomial growth of degree $d$}
if the cardinality of balls is trapped between $(\log r)^{-C} r^d$ and $(\log r)^C r^d$
for some universal constant $C \geq 1$.

Planar graphs of uniform (or nearly-uniform) polynomial volume growth of degree $d > 2$ arise
in a number of contexts.  In particular, they appear in the study of random triangulations in 2D quantum gravity \cite{adj97} and
as combinatorial approximations to the boundaries of $3$-dimensional hyperbolic groups in geometric group theory (see, e.g., \cite{bk02}).

When the dimension of volume growth disagrees with the topological dimension, one sometimes witnesses
certain geometrically or spectrally degenerate behaviors.  For instance, it is known 
that random planar triangulations of the $2$-sphere have nearly-uniform polynomial
volume growth of degree $4$ (in an appropriate statistical, asymptotic sense) \cite{angel03}.
The distributional limit (see \pref{sec:graph-limits}) of such graphs is called the uniform
infinite planar triangulation (UIPT).
But this $4$-dimensional volume growth does not come
with $4$-dimensional isoperimetry:  With high probability, a ball in the UIPT of radius $r$ about a vertex $v$
can be separated from the complement of a $2r$ ball about $v$ by removing a set of size $O(r)$.
And, indeed, Benjamini and Papasoglu \cite{BP11} showed that this phenomenon holds generally:
such annular separators of size $O(r)$ exist in all planar graphs with uniform polynomial volume growth.

Similarly, it is known that diffusion on the UIPT is {\em anomalous.}
Specifically, the random walk on the UIPT is almost surely subdiffusive.
In other words, if $\{X_t\}$ is the random walk and $d_G$ denotes the graph metric,
then $\E d_G(X_0,X_t) \leq t^{1/2-\eps}$ for some $\eps > 0$.
This was established by Benjamini and Curien \cite{bc13}.  In \cite{lee17a}, it is shown that
on {\em any} unimodular random planar graph with nearly-uniform polynomial growth of degree $d > 3$ (in a suitable
statistical sense), the random walk is subdiffusive.
So again, a disagreement between the dimension of volume growth
and the topological dimension results in a degeneracy typical
in the geometry of fractals (see, e.g., \cite{Barlow98}).

Finally, consider a seminal result of Benjamini and Schramm \cite{bs01}:  If $(G,\rho)$ is
the local distributional limit of a sequence of finite planar graphs
with uniformly bounded degrees, then $(G,\rho)$ is almost surely recurrent.
In this sense, any such limit is spectrally (at most) two-dimensional.
This was extended by Gurel-Gurevich and Nachmias \cite{gn13} to unimodular random
graphs with an exponential tail on the degree of the root, making it applicable to the UIPT.
Benjamini \cite{BenStFlour13} has conjectured that this holds
for every planar graph with uniform polynomial volume.
We construct a family of counterexamples.
Our focus on rational degrees of growth is largely for simplicity;
suitable variants of our construction should yield similar results for all real $d > 2$
(see \pref{rem:real-degrees}).

\begin{theorem}\label{thm:transient-intro}
   For every rational $d > 2$, there is a {\em transient} planar graph with uniform
   polynomial growth of degree $d$.
\end{theorem}

Conversely, it is well-known that {\em any graph} with growth rate $d \leq 2$ is recurrent.
The examples underlying \pref{thm:transient-intro} cannot be unimodular.
Nevertheless, we construct unimodular examples addressing some of the issues raised above.
Angel and Nachmias (unpublished) showed the existence, for every $\e > 0$ sufficiently small,
of a unimodular random planar graph $(G,\rho)$ on which the random walk
is almost surely diffusive, and which almost surely satisfies
\[
   \lim_{r \to \infty} \frac{\log |B_G(\rho,r)|}{\log r} = 3 - \e.
\]
Here, $B_G(\rho,r)$ is the graph ball around $\rho$ of radius $r$.
In other words, $r$-balls have an asymptotic growth rate of $r^{3-\e}$ as $r \to \infty$.

The authors of \cite{BP11} asked whether in planar graphs with uniform growth of degree $d \geq 2$,
the speed of the walk should be at most $t^{1/d+o(1)}$.  We recall the following weaker theorem.

\begin{theorem}[\cite{lee17a}]
   Suppose $(G,\rho)$ is a unimodular random planar graph and $G$ almost surely
   has uniform polynomial growth of degree $d$.  Then:
   \[
      \E\left[d_G(X_0,X_t) \mid X_0 = \rho\right] \lesssim t^{1/\max(2,d-1)}.
   \]
\end{theorem}

We construct examples where this dependence is nearly tight.

\begin{theorem}\label{thm:speed-intro}
   For every rational $d \geq 2$ and $\e > 0$, there is a constant $c(\e) > 0$ and
   a unimodular random planar graph $(G,\rho)$ such that $G$ almost surely has uniform polynomial growth of degree
   $d$, and
   \[
      \E\left[d_G(X_0,X_t) \mid X_0 = \rho\right] \geq c(\e) t^{1/\left(\max(2,d-1)+\e\right)}.
   \]
\end{theorem}

Finally, let us address another question from \cite{BP11}.
In conjunction with the existence of small annular separators,
the authors asked whether a planar graph with uniform polynomial growth
of degree $d > 2$ can be such that the complement of every ball is connected.
For example, in the UIPT, there are ``baby universes'' connected
to the graph via a thin neck that can be cut off by removing
a small graph ball.  

\begin{theorem}\label{thm:intro-complement}
   For every rational $d \geq 2$, there is a unimodular random planar graph $(G,\rho)$
   such that almost surely:
   \begin{enumerate}
      \item $G$ has uniform polynomial growth of degree $d$.
      \item The complement of every graph ball in $G$ is connected.
   \end{enumerate}
\end{theorem}

\medskip
\noindent
{\bf Annular resistances.}
Our unimodular constructions have the property that the ``Einstein relations'' (see, e.g., \cite{Barlow98})
for various dimensional exponents do not hold.
In particular, this implies that the graphs we construct are not strongly recurrent (see, e.g., \cite{KM08}).
Indeed, the effective resistance across annuli can be made very small
(see \pref{sec:resist} for the definition of effective resistance).

\begin{theorem}\label{thm:intro-resist}
   For every $\e > 0$ and $d \geq 3$, there is a unimodular random planar graph $(G,\rho)$
   that almost surely has uniform polynomial volume growth of degree $d$ and, moreover,
   almost surely satisfies
   \begin{equation}\label{eq:intro-reff}
      \reff^G\left(B_G(\rho,R) \leftrightarrow V(G) \setminus B_G(\rho,2R)\right) \leq C(\e) R^{-(1-\e)}, \quad \forall R \geq 1,
   \end{equation}
   where $C(\e) \geq 1$ is a constant depending only on $\e$.
\end{theorem}

Note that the existence of annular separators of size $O(R)$ mentioned previously
gives $\reff^G\left(B_G(\rho,R) \leftrightarrow V(G) \setminus B_G(\rho,2R)\right) \gtrsim R^{-1}$
by the Nash-Williams inequality.
Moreover, recall that since the graph $(G,\rho)$ from \pref{thm:intro-resist} is unimodular and planar,
it must be almost surely recurrent (cf. \cite{bs01}).
Therefore the electrical flow witnessing \eqref{eq:intro-reff} cannot spread out ``isotropically''
from $B_G(\rho,R)$ to $B_G(\rho,2R)$.
Indeed, if one were able to send a flow roughly uniformly from $B_G(\rho,2^i)$ to $B_G(\rho,2^{i+1})$, then
these electrical flows would chain to give
\[
   \reff^G\left(\rho \leftrightarrow V(G) \setminus B_G(\rho,2^i)\right) \lesssim \sum_{j \leq i} 2^{-(1-\e)j},
\]
and taking $i \to \infty$ would show that $G$ is transient.

One formalization of this fact is that the graphs in \pref{thm:intro-resist} (almost surely) do not
satisfy an elliptic Harnack inequality.
These graphs are almost surely one-ended, and one can easily pass to a quasi-isometric
triangulation that admits a circle packing whose carrier is the entire plane $\R^2$.
By a result of Murugan \cite{Murugan19}, this implies that
the graph metric $(V(G),d_G)$ on the graphs in \pref{thm:intro-resist}
is {\em not} quasisymmetric to the Euclidean metric induced
on the vertices by any such circle packing.
(This can also be proved directly from \eqref{eq:intro-reff}.)

We remark on one other interesting feature of \pref{thm:intro-resist}.
Suppose that $\Gamma$ is a Gromov hyperbolic group whose visual boundary $\partial_{\infty} \Gamma$
is homeomorphic to the $2$-sphere $\bbS^2$.
The authors of \cite{bk02} construct a family $\{G_n : n \geq 1\}$ of discrete approximations to $\partial_{\infty} \Gamma$
such that each $G_n$ is a planar graph and the family $\{G_n\}$ has uniform polynomial volume growth.\footnote{More precisely,
for the boundary of a hyperbolic group as above, one can choose a sequence of approximations with this property.}
They show that if there is a constant $c > 0$
so that the annuli in $G_n$ satisfy uniform effective resistance estimates of the form
\[
   \reff^{G_n}\left(B_{G_n}(x,R) \leftrightarrow V(G_n) \setminus B_{G_n}(x,2R)\right) \geq c, \quad \forall 1 \leq R \leq \diam(G_n)/10,\ x \in V(G_n),\ \forall n \geq 1,
\]
then $\partial_{\infty} \Gamma$ is quasisymmetric to $\bbS^2$ (cf. \cite[Thm 11.1]{bk02}.)

In particular, if it were to hold that for any (infinite) planar graph $G$ with uniform polynomial growth we have
\[
   \reff^{G}\left(B_{G}(x,R) \leftrightarrow V(G) \setminus B_{G}(x,2R)\right) \geq c > 0, \quad \forall R \geq 1, x \in V(G),
\]
then it would confirm positively Cannon's conjecture from geometric group theory.
\pref{thm:intro-resist} exhibits graphs for which this fails in essentially the strongest way possible.

\subsection{Preliminaries}

We will consider primarily connected, undirected graphs $G=(V,E)$,
which we equip with the associated path metric $d_G$.
We will sometimes write $V(G)$ and $E(G)$, respectively, for the vertex and edge
sets of $G$.
If $U \subseteq V(G)$, we write $G[U]$ for the subgraph induced on $U$.

For $v \in V$, let $\deg_G(v)$ denote the degree of $v$ in $G$.
Let $\diam(G) \seteq \sup_{x,y \in V} d_G(x,y)$ denote the diameter
(which is only finite for $G$ finite).
For $v \in V$ and $r \geq 0$, we use $B_G(v,r) = \{ u \in V : d_G(u,v) \leq r\}$
to denote the closed ball in $G$.
For subsets $S,T \subseteq V$, we write $d_G(S,T) \seteq \inf \{ d_G(s,t) : s \in S, t \in T\}$.

Say that an infinite graph $G$ has {\em uniform volume growth of rate $f(r)$}
if there exist constants $C,c > 0$ such that
\[
   c f(r) \leq |B_G(v,r)| \leq C f(r)\qquad \forall v \in V, r \geq 1.
\]
A graph has {\em uniform polynomial growth of degree $d$} if it has uniform
volume growth of rate $f(r) = r^d$, and has {\em uniform polynomial growth}
if this holds for some $d > 0$.

For two expressions $A$ and $B$, we use the notation $A \lesssim B$ to denote
that $A \leq C B$ for some {\em universal} constant $C$.  The notation
$A \lesssim_{\gamma} B$ denotes that $A \leq C(\gamma) B$ where $C(\gamma)$
is a number depending only on the parameter $\gamma$.
We write $A \asymp B$ for the conjunction $A \lesssim B \wedge B \lesssim A$.

\subsubsection{Distributional limits of graphs}
\label{sec:graph-limits}

We briefly review the weak local topology on random rooted graphs.
One may consult the extensive reference of Aldous and Lyons \cite{aldous-lyons},
and \cite{bc12} for the corresponding theory of reversible random graphs.
The paper \cite{bs01}
offers a concise introduction to distributional limits of finite planar graphs.
We briefly review some relevant points.

Let $\graphs$ denote the set of isomorphism classes of connected, locally finite graphs;
let $\rgraphs$ denote the set of {\em rooted} isomorphism classes of {\em rooted}, connected,
locally finite graphs.
Define a metric on $\rgraphs$ as follows:  $\dloc\left((G_1,\rho_1), (G_2,\rho_2)\right) = 1/(1+\alpha)$,
where
\[
   \alpha = \sup \left\{ r > 0 : B_{G_1}(\rho_1, r) \cong_{\rho} B_{G_2}(\rho_2, r) \right\}\,,
\]
and we use $\cong_{\rho}$ to denote rooted isomorphism of graphs.
$(\rgraphs, \dloc)$ is a separable, complete metric space. For probability measures $\{\mu_n\}, \mu$ on
$\rgraphs$,
write $\{\mu_n\} \Rightarrow \mu$ when $\mu_n$ converges weakly to $\mu$ with respect to $\dloc$.

A random rooted graph $(G,X_0)$ is said to be {\em reversible}
if $(G,X_0,X_1)$ and $(G,X_1,X_0)$ have the same law, where $X_1$ is a uniformly
random neighbor of $X_0$ in $G$.
A random rooted graph $(G,\rho)$ is said to be {\em unimodular} if it satisfies the Mass Transport Principle (see, e.g., \cite{aldous-lyons}).
For our purposes, it suffices to note that if
$\E[\deg_G(\rho)] < \infty$, then
$(G,\rho)$ is unimodular if and only if the random rooted graph $(\tilde{G},\tilde{\rho})$ is reversible,
where $(\tilde{G},\tilde{\rho})$ has the law of $(G,\rho)$ biased by $\deg_G(\rho)$

If $\{(G_n,\rho_n)\} \todl (G,\rho)$,
we say that $(G,\rho)$ is the {\em distributional limit} of the sequence $\{(G_n,\rho_n)\}$,
where we have conflated random variables with their laws in the obvious way.
Consider a sequence $\{G_n\} \subseteq \graphs$ of finite graphs,
and let $\rho_n$ denote a uniformly random element of $V(G_n)$.  Then $\{(G_n,\rho_n)\}$
is a sequence of $\rgraphs$-valued random variables,
and one has the following: if $\{(G_n,\rho_n)\} \todl (G,\rho)$, then $(G,\rho)$ is unimodular.
Equivalently, 
if $\{(G_n,\rho_n)\}$ is a sequence of connected finite graphs
and $\rho_n \in V(G_n)$ is chosen according to the stationary measure of $G_n$,
then if $\{(G_n,\rho_n)\} \implies (G,\rho)$, it holds that $(G,\rho)$ is a
reversible random graph.

\section{A transient planar graph of uniform polynomial growth}

We begin by constructing a transient planar graph with uniform polynomial growth of degree $d > 2$.
Our construction in this section has $d = \log_3(12) \approx 2.26$.
In \pref{sec:generalizations}, this construction is generalized to any rational $d > 2$.

\subsection{Tilings and dual graphs}
\label{sec:tilings}

\begin{figure}[h]
      \begin{center}
         \subfigure[A tiling of the unit square\label{fig:tiling}]{ \includegraphics[width=4.5cm]{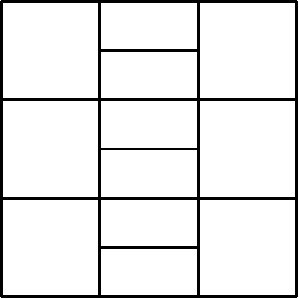}} \hspace{0.8in}
         \subfigure[The associated dual graph\label{fig:dual}]{  \includegraphics[width=4.5cm]{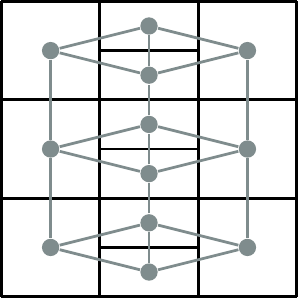}}
         \caption{Tilings and their dual graph}
   \end{center}
\end{figure}

Our constructions are based on planar tilings by rectangles.
A {\em tile} is an axis-parallel closed rectangle $A \subseteq \R^2$.
We will encode such a tile as a triple $(p(A), \ell_1(A), \ell_2(A))$, where
$p(A) \in \R^2$ denotes its bottom-left corner, $\ell_1(A)$ its width (length of
its projection onto the $x$-axis), and $\ell_2(A)$ its height (length of its projection onto the $y$-axis).
A {\em tiling $\bT$} is a finite collection of interior-disjoint tiles.
Denote $\region{\bT} \seteq \bigcup_{A \in \bT} A$.  If $R \subseteq \R^2$, we say that
{\em $\bT$ is a tiling of $R$} if $\region{\bT} = R$.
See \pref{fig:tiling} for a tiling of the unit square.

We associate to a tiling its {\em dual graph} $G(\bT)$ with vertex set $\bT$
and with an edge between two tiles $A,B \in \bT$ whenever $A \cap B$ has
Hausdorff dimension one; in other words, $A,B$ are tangent, but not only at a corner.
Denote by $\cT$ the set of all tilings of the unit square.
See \pref{fig:dual}.
For the remainder of the paper, we will consider only tilings $\bT$ for which $G(\bT)$ is connected.

\begin{figure}
\centering
	\includegraphics{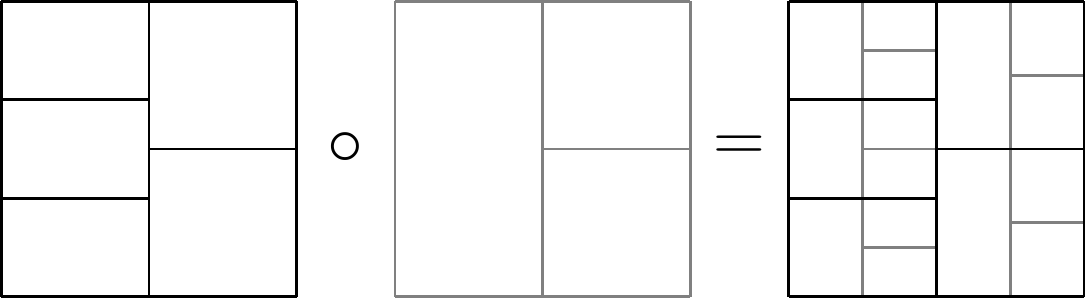}
   \caption{An example of the tiling product $S \tileprod T$\label{fig:tileprod}}
\end{figure}

\begin{definition}[Tiling product]
For $\bS, \bT \in \cT$, define the product
$\bS\tileprod \bT \in \cT$ as the tiling formed by replacing every tile in
$\bS$ by an (appropriately scaled) copy of $\bT$.
More precisely:  For every $A \in \bS$ and $B \in \bT$, there is a tile $R \in \bS\tileprod\bT$
with $\ell_i(R) \seteq \ell_i(A) \ell_i(B)$, and
\[
   p_i(R) \seteq p_i(A)+p_i(B) \ell_i(A),
\]
for each $i \in \{1,2\}$.
See \pref{fig:tileprod}.
\end{definition}

If $\bT \in \cT$ and $n \geq 0$, we will use $\bT^n \seteq \bT \tileprod \cdots \tileprod \bT$ to denote
the $n$-fold tile product of $\bT$ with itself.  The following observation
shows that this is well-defined.

\begin{observation}
   The tiling product is associative:  $(\bS \tileprod \bT) \tileprod \bU = \bS \tileprod (\bT \tileprod \bU)$
   for all $S,T,U \in \cT$.
   Moreover, if $\bI \in \cT$ consists of the single tile $[0,1]^2$, then $\bT \tileprod \bI = \bI \tileprod \bT$
   for all $\bT \in \cT$.
\end{observation}

\begin{figure}
\centering
	\includegraphics{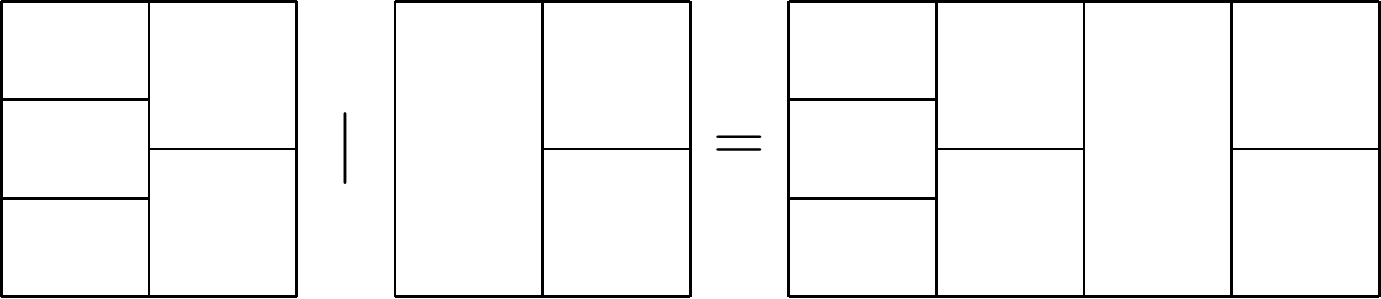}
   \caption{An example of the tiling concatenation $S \tilecat T$\label{fig:tilecat}}
\end{figure}

\begin{definition}[Tiling concatenation]
   \label{def:concat}
   Suppose that $\bS$ is a tiling of a rectangle $R$ and $\bT$ is a tiling of a rectangle $R'$
   and the heights of $R$ and $R'$ coincide.
   Let $R''$ denote the translation of $R'$ for which the left edge of $R''$ coincides with the right edge of $R$,
   and denote by $\bS \tilecat \bT$ the induced tiling of the rectangle $R \cup R''$.
   See \pref{fig:tilecat}.
\end{definition}

\begin{figure}[h]
   \begin{center}
   \includegraphics[width=16cm]{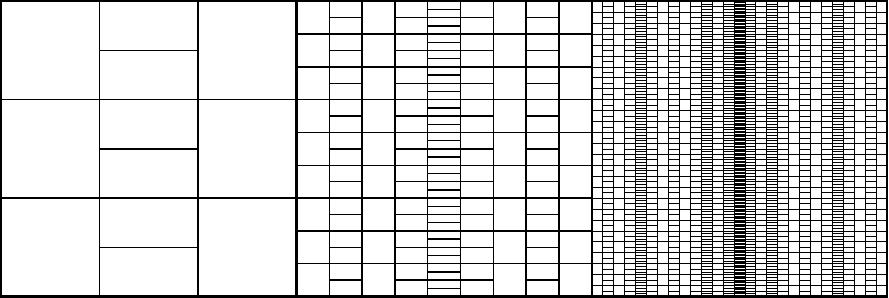}%
   \end{center}
   \caption{The tiling $\bH^1 \tilecat \bH^2 \tilecat \bH^3$\label{fig:transient}}
\end{figure}

Let $\bH$ denote the tiling in \pref{fig:tiling}, and define
$\cH_n \seteq G(\bH^0 \tilecat \bH^1 \tilecat \cdots \tilecat \bH^n)$;
see \pref{fig:transient}, where we have omitted $\bH^0$ for ease of illustration.
The next theorem represents our primary goal for the remainder of this section.
Note that $\bH^0 = \{\rho\}$ consists of a single tile,
and that $\{(\cH_n,\rho)\}$ forms a Cauchy sequence in $(\rgraphs, \dloc)$,
since $(\cH_n,\rho)$ is naturally a rooted subgraph of $(\cH_{n+1},\rho)$.
Letting $\cH_{\infty}$ denote its limit, we will establish the following.

\begin{theorem}\label{thm:transient}
   The infinite planar graph $\cH_{\infty}$ is transient and has uniform polynomial volume growth of degree $\log_3(12)$.
\end{theorem}

Uniform growth is established in \pref{lem:Hn-uniform-growth} and transience in \pref{cor:transient}.

\subsection{Volume growth}

The following lemma shows that a ball of radius $r = \diam(\bH^n)$ in $\bH^n$ has volume $\asymp r^{\log_3(12)}$. Later on, in \pref{lem:growthHn}, we will show a similar bound holds for balls of arbitrary radius $1 \leq r \leq \diam(\bH^n)$ in $\bH^n$.
\begin{lemma}
\label{lem:example_diam}
   For $n \geq 0$, we have $|\bm{H}^n| = 12^n$, and $3^n \leq \diam(\bH^n) \leq 3^{n+1}$.
\end{lemma}

\begin{proof}
   The first claim is straightforward by induction. 
   For the second claim, note that $\ell_1(A) = 3^{-n}$ for every $A \in \bH^n$.
      Moreover, there are $3^n$ tiles touching the left-most boundary of $[0,1]^2$.
   Therefore to connect any $A,B \in \bH^n$ by a path in $G(\bH^n)$, we need only go from $A$
   to the left-most column in at most $3^n$ steps, then use at most $3^n$ steps of the column,
   and finally move at most $3^n$ steps to $B$.
\end{proof}

The next lemma is straightforward.

\begin{lemma}
\label{lem:contract_growth_lower_bound}
Consider $\bS,\bT \in \cT$ and $G=G(\bS\tileprod\bT)$.
For any $X \in \bS\tileprod\bT$, it holds that $|B_G(X, \diam(G(\bT)))| \geq |\bT|$.
\end{lemma}

If $\bT$ is a tiling, let us partition the edge set $E(G(\bT)) = E_1(\bT) \cup E_2(\bT)$ into horizontal and vertical edges.
For $A \in \bT$ and $i \in \{1,2\}$, let $N_{\bT}(A,i)$ denote the set of tiles adjacent to $A$ in $G(\bT)$ along the $i$th direction, meaning that the edge separating them is parallel to the $i$th axis (see \pref{fig:adjacent}).

\begin{figure}[h]
	\begin{center}
		\includegraphics[width=4.5cm]{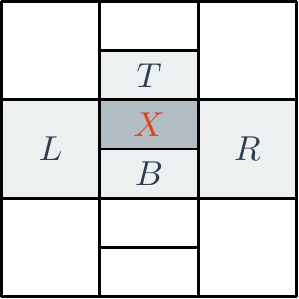}%
	\end{center}
   \caption{A tile $X \in \bH$ and its neighbors are marked. Here we have $N_{\bH}(X, 1) = \{T, B\}$ and $N_{\bH}(X, 2) = \{L, R\}$.\label{fig:adjacent}}
\end{figure}

Further denote $N_{\bT}(A) \seteq N_{\bT}(A,1) \cup N_{\bT}(A,2)$.
Moreover, we define:
\begin{align}
   \alpha_{\bT}(A,i) &\seteq \max \left\{ \frac{\ell_{j}(A)}{\ell_{j}(B)}: B \in N_{\bT}(A,i), j \in \{1,2\} \right\}, \quad i \in \{1,2\} \nonumber \\
   \alpha_{\bT}(A) &\seteq \max_{i \in \{1,2\}} \alpha_{\bT}(A,i) \nonumber\\
   \alpha_{\bT} &\seteq \max \left\{ \alpha_{\bT}(A) : A \in \bT\right\} \nonumber\\
   L_{\bT} &\seteq \max \left\{ \ell_i(A) : A \in \bT, i \in \{1,2\}\right\}. \label{eq:def-L}
\end{align}
We take $\alpha_{\bT} \seteq 1$ if $\bT$ contains a single tile.
It is now straightforward to check that $\alpha_{\bT}$ bounds the degrees in $G(\bT)$.

\begin{lemma}
\label{lem:degree_bound}
   For a tiling $\bT$ and $A \in \bT$, it holds that
   \[
      \deg_{G(\bT)}(A) \leq 4(1+\alpha_{\bT}) \leq 8 \alpha_{\bT}.
   \]
\end{lemma}
\begin{proof}
   After accounting for the four corners of $A$, every other tile $B \in N_{\bT}(A,i)$
   intersects $A$ in a segment of length at least $\ell_i(B) \geq \ell_i(A)/\alpha_{\bT}$.
   The second inequality follows from $\alpha_{\bT} \geq 1$.
\end{proof}

\begin{lemma}
\label{lem:growth_upper_bound}
Consider $\bS,\bT \in \cT$ and let $G = G(\bS \tileprod \bT)$.
Then for any $X \in \bS \tileprod \bT$, it holds that
\begin{equation}\label{eq:growth}
      |B_G(X,1/(\alpha^4_{\bS} L_{\bT}))| \leq 192 \alpha^2_{\bS} |\bT|.
\end{equation}
\end{lemma}

\begin{proof}
   For a tile $Y \in \bS \tileprod \bT$, let $\hat{Y} \in \bS$ denote
   the unique tile for which $Y \subseteq \hat{Y}$.
   Let us also define
   \[
      \tilde{N}_{\bS}(\hat{X}) \seteq \{\hat{X}\} \cup N_{\bS}(\hat{X}) \cup N_{\bS}\left(N_{\bS}(\hat{X},1),2\right) \cup N_{\bS}\left(N_{\bS}(\hat{X},2),1\right),
   \]
   which is the set of vertices of $G(\bS)$ that can be reached from $\hat{X}$ by following at most one edge in each direction.

   We will show that
   \begin{equation}\label{eq:regions}
      \region{B_G(X,1/(\alpha^4_{\bS} L_{\bT}))}  \subseteq \region{\tilde{N}_{\bS}(\hat{X})}.
   \end{equation}
   It follows that 
   \[
      |B_G(X,1/(\alpha^4_{\bS} L_{\bT}))| \leq |\bT| \cdot |\tilde{N}_{\bS}(\hat{X})| \leq |\bT| \cdot  3\left(\max_{A \in \bS} \deg_{G(\bS)}(A)\right)^2,
   \]
   and then \eqref{eq:growth} follows from \pref{lem:degree_bound}.

   To establish \eqref{eq:regions}, consider any path $\llangle X=X_0,X_1,X_2,\ldots,X_h\rrangle$ in $G$
   with $\hat{X}_h \notin \tilde{N}_{\bS}(\hat{X})$.
   Let $k \leq h$ be the smallest index for which $\hat{X}_k \notin \tilde{N}_{\bS}(\hat{X})$.
   Then:
   \begin{align}
      &{X_0},{X_1},\ldots,{X_{k-1}} \subseteq \region{\tilde N_{\bS}(\hat{X})}\label{eq:gp1} \\
      &X_{k-1} \cap \left(\partial \region{\tilde N_{\bS}(\hat{X})} \cap (0,1)^2\right) \neq \emptyset.\label{eq:gp2}
   \end{align}
   Now \eqref{eq:gp1} implies that
   \begin{equation}\label{eq:gp3}
      \ell_i(X_j) \leq L_{\bT} \ell_i(\hat{X}_j) \leq L_{\bT} \alpha^2_{\bS} \ell_i(\hat{X}),\qquad j \leq k-1, i \in \{1,2\}.
   \end{equation}
   And \eqref{eq:gp2} shows that for some $i' \in \{1,2\}$,
   \begin{equation}\label{eq:gp4}
      \sum_{j=0}^{k-1} \ell_{i'}(X_j) \geq \min \left\{ \ell_{i'}(Y) : Y \in \tilde{N}_{\bm{S}}(\hat{X}) \right\} \geq \ell_{i'}(\hat{X})/\alpha^2_{\bS}.
   \end{equation} 
   To clarify why this is true, note that 
   \[
\hat{X} + \left[-\ell_1(\hat{X})/\alpha^2_{\bS}, \ell_1(\hat{X})/\alpha^2_{\bS}\right] \times \left[-\ell_2(\hat{X})/\alpha_{\bS}^2, \ell_2(\hat{X})/\alpha^2_{\bS}\right] \subseteq \region{\tilde{N}_{\bS}(\hat{X})},
   \]
   where '$+$' here is the Minkowski sum $R+S \seteq \{ r + s : r \in R, s \in S \}$.
   Indeed, this inclusion motivates our definition of the ``$\ell_{\infty}$ neighborhood'' $\tilde{N}$ above.

   Combining \eqref{eq:gp3} and \eqref{eq:gp4} now gives %
   \[
      h \geq k \geq \frac{1}{\alpha_{\bS}^4 L_{\bT}},
   \]
   completing the proof.
\end{proof}

We can now finish the analysis of the volume growth in the graphs $\{\bH^n : n \geq 0\}$. 
\begin{lemma} 
\label{lem:example_Delta_bound}
   For $n \geq 1$, it holds that
      $\alpha_{\bH^n} \leq 2$. 
\end{lemma}
\begin{proof}
Consider $A \in \bH^n$ and $B \in N_{\bH^n}(A)$.
First note that, as in the proof of \pref{lem:example_diam}, all the tiles in $\bH^n$ have the same width $3^{-n}$, and so $\ell_1(A) = \ell_1(B)$.
Moreover, one can easily verify that every two vertically adjacent tiles in $\bH^n$ have the same height, and so we have $\ell_2(A)=\ell_2(B)$ when $B \in N_{\bH^n}(A, 1)$. Now we prove by an induction on $n$ that for all horizontally adjacent tiles $A, B \in \bH^n$ we have 
\[
\frac{\ell_2(A)}{\ell_2(B)} \leq 2.
\]

The base case is clear for $n = 1$. For $n \geq 2$
Let us write $\bH^n = \bH \tileprod \bH^{n-1}$, and let $\hat{A},\hat{B} \in \bH$ be the unique tiles for which $A \subseteq \hat{A}$
   and $B \subseteq \hat{B}$.
   If $\hat{A}=\hat{B}$, then the claim follows from the induction hypothesis. Otherwise, as $B \in N_{\bH^n}(A, 2)$, it holds that $\hat{B} \in N_{\bH}(\hat{A}, 2)$ as well. By symmetry, the tiles touching the left and right edges of $\bH^{n-1}$ have the same height, and therefore it follows that 
	\[
	\frac{\ell_2(A)}{\ell_2(B)} \leq \frac{\ell_2(\hat{A})}{\ell_2(\hat{B})} \leq 2,
	\]
	completing the proof.
\end{proof}

\begin{lemma}
   \label{lem:growthHn}
   For any $n \geq 0$, it holds that
   \[
      \left|B_{G(\bH^n)}(A,r)\right| \asymp r^{\log_3(12)},\qquad \forall A \in \bH^n, 1 \leq r \leq \diam(G(\bH^n)).
   \]
\end{lemma}

\begin{proof}
   Writing $\bH^n = \bH^{n-k} \tileprod \bH^{k}$ and employing \pref{lem:contract_growth_lower_bound} together with \pref{lem:example_diam}
   gives
   \[
      \left|B_{G(\bH^n)}(A, 3^{k+1})\right| \geq |\bH^k| = 12^k,\qquad \forall A \in \bH^n, k \in \{0,1,\ldots,n\}.
   \]
   The desired lower bound now follows using monotonicity of $|B_{G(H^n)}(A,r)|$ with respect to $r$.

   To prove the upper bound, first note that we have $L_{\bH^k} = 3^{-k}$ (recall the definition \eqref{eq:def-L}). Moreover, by \pref{lem:example_Delta_bound} we have $\alpha_{\bH^n} \leq 2$. Hence invoking \pref{lem:growth_upper_bound} with $\bS = \bH^{n-k}$ and $\bT = \bH^k$ gives
   \[
      \left|B_{G(\bH^n)}(A, 3^{k}/16)\right| \leq 768 \cdot 12^{k},\qquad \forall A \in \bH^n, k \in \{0,1,\ldots,n\},
   \]
   completing the proof.
\end{proof}

Finally, this allows us to establish a uniform polynomial growth rate for $\cH_{\infty}$.

\begin{lemma}
   \label{lem:Hn-uniform-growth}
   It holds that
   \[
      |B_{\cH_{\infty}}(v,r)| \asymp r^{\log_3(12)}\qquad \forall v \in V(\cH_{\infty}), r \geq 1.
   \]
\end{lemma}

\begin{proof}
   Recall first the natural identification $\bH^n \hookrightarrow V(\cH_{\infty})$
   under which $V(\cH_{\infty}) = \bigcup_{n \geq 0} \bH^n$ is a partition.
   Consider $v \in V(\cH_{\infty})$ and let $n \geq 0$ be such that $v \in\bH^n$.
   Now \pref{lem:growthHn} in conjunction with \pref{lem:example_diam} yields the bounds:
   \begin{align*}
      |B_{\cH_{\infty}}(v,r)| &\geq |B_{\cH_{\infty}}(v,r) \cap \bH^n| \gtrsim r^{\log_3(12)} && r \leq 3^{n+3} \\
      |B_{\cH_{\infty}}(v,r)| &\geq |\bH^k| = 12^k \gtrsim r^{\log_3(12)} && r \in [3^{k+3},3^{k+4}), k \geq n \\
      |B_{\cH_{\infty}}(v,r)| &=
      |B_{\cH_{\infty}}(v,r) \cap \bH^{n-1}| + 
      |B_{\cH_{\infty}}(v,r) \cap \bH^n| + 
      |B_{\cH_{\infty}}(v,r) \cap \bH^{n+1}| \\
                              &\lesssim r^{\log_3(12)} && r \leq 3^{n-1} \\
      |B_{\cH_{\infty}}(v,r)| &\leq \sum_{j \leq \max(k,n+1)} |\bH^j| \leq 2\cdot 12^{\max(k,n+1)} \lesssim r^{\log_3(12)} && r \in [3^{k-1},3^k), k \geq n.
   \end{align*}
   These four bounds together verify the desired claim.
\end{proof}

\subsection{Effective resistances}
\label{sec:resist}

Consider a weighted, undirected graph $G=(V,E,c)$ with edge conductances $c :E \to \R_+$.
For $p \geq 1$,
denote $\ell_p(V) \seteq \{ f : V \to \R \mid \sum_{u \in V} |f(u)|^p < \infty \}$,
and equip $\ell_2(V)$ with the inner product
$\langle f,g\rangle = \sum_{u \in V} f(u) g(u)$.

For $s,t \in \ell_1(V)$ with $\|s\|_1=\|t\|_1$,
we define the {\em effective resistance}
\[
   \reff^G(s,t) \seteq \llangle s-t, L_G^{\dagger} (s-t)\rrangle,
\]
where $L_G$ is the combinatorial Laplacian of $G$, and $L_G^{\dagger}$ is the Moore-Penrose pseudoinverse.
Here, $L_G$ is the operator on $\ell_2(V)$ defined by
\[
   L_G f(v) = \sum_{u : \{u,v\} \in E} c(\{u,v\}) \left(f(v)-f(u)\right).
\]
If $G$ is unweighted, we assume it is equipped with unit conductances $c \equiv \1_{E(G)}$.

Equivalently, if we consider mappings $\theta : E \to \R$, and define the energy functional
\[
   \cE_G(\theta) \seteq \sum_{e \in E} c(e)^{-1} \theta(e)^2,
\]
then $\reff^G(s,t)$ is the minimum energy of a flow with demands $s-t$.
(See, for instance, \cite[Ch. 2]{lp:book}.)
For two finite sets $A,B \subseteq V$ in a graph, we define
\[
   \reff^G(A \leftrightarrow B) \seteq \inf \left\{ \reff^G(s,t) : \supp(s) \subseteq A, \supp(t) \subseteq B, s,t \in \ell_1(V), \|s\|_1=\|t\|_1=1\right\},
\]
and we recall the following standard characterization (see, e.g., \cite[Thm. 2.3]{lp:book}).

If we define additionally $c_v \seteq \sum_{u \in V(L)} c(\{u,v\})$ for $v \in V$,
then we can recall that weighted random walk $\{X_t\}$ on $G$ with Markovian law
\[
   \Pr\left[X_{t+1}=v \mid X_t=u\right] = \frac{c(\{u,v\})}{c_u},\qquad u,v \in V.
\]

\begin{theorem}[Transience criterion]\label{thm:bdd-energy}
   A weighted graph $G=(V,E,c)$ is transient if and only if there is a vertex $v \in V$
   and an increasing sequence $V_1 \subseteq \cdots \subseteq V_n \subseteq V_{n+1} \subseteq \cdots$
   of finite subsets of vertices satisfying $\bigcup_{n \geq 1} V_n = V$
   and 
   \[
      \sup_{n \geq 1} \reff^G\left(\{v\} \leftrightarrow V \setminus V_n\right) < \infty.
   \]
\end{theorem}

For a tiling $\bT$ of a closed rectangle $R$, let $\scrL(\bT)$ and $\scrR(\bT)$ denote the sets
of tiles that intersect the left and right edges of $R$, respectively.
We define
\[
   \rho(\bT) \seteq \reff^{G(\bT)}\left(\1_{\scrL(\bT)}/|\scrL(\bT)|, \1_{\scrR(\bT)}/|\scrR(\bT)|\right),
\]

\begin{observation}\label{obs:LR}
   For any $\bS,\bT \in \cT$, we have $|\scrL(\bS \tileprod \bT)| = |\scrL(\bS)|\cdot |\scrL(\bT)|$ and
   $|\scrR(\bS \tileprod \bT)| = |\scrR(\bS)|\cdot |\scrR(\bT)|$.
   In particular, $|\scrL(\bH^n)| = |\scrR(\bH^n)| = 3^n$.
\end{observation}

\begin{lemma}
\label{lem:side_product_reff}
Suppose that $\bS,\bT$ are tilings satisfying the conditions of \pref{def:concat}.
Suppose furthermore that all rectangles in $\scrR(\bS)$ have the same height,
and the same is true for $\scrL(\bT)$.  Then we have
\[
   \rho(\bS \tilecat \bT) \leq \rho(\bS) + \rho(\bT) + \frac{1}{\max(|\scrR(\bS)|,|\scrL(\bT)|)}\,.
\]
\end{lemma}
\begin{proof}
   By the triangle inequality for effective resistances, it suffices to prove that
   \[
      \reff^{G}\left(\1_{\scrR(\bS)}/|\scrR(\bS)|, \1_{\scrL(\bT)}/|\scrL(\bT)|\right) \leq \frac{1}{\max(|\scrR(\bS)|,|\scrL(\bT)|)}\,,
   \]
   where $G = G(\bS \tilecat \bT)$.
   We construct a flow from $\scrR(\bS)$ to $\scrL(\bT)$ as follows:  If $A \in \scrR(\bS), B \in \scrL(\bT)$
   and $\{A,B\} \in E(G)$, then the flow value on $\{A,B\}$ is 
   \[
      F_{AB} \seteq \frac{\mathrm{len}(A \cap B)}{\ell_2(A)} \frac{1}{|\scrR(\bS)|}.
   \] 
   Denoting $m \seteq \max(|\scrR(\bS)|,|\scrL(\bT)|)$, we clearly have
   $F_{AB} \leq 1/m$.  Moreover,
   \[
      \sum_{A \in \scrR(\bS)} \sum_{\substack{B \in \scrL(\bT) : \\ \{A,B\} \in E(G)}} F_{AB} = 1,
   \]
   hence
   \[
      \sum_{A \in \scrR(\bS)} \sum_{\substack{B \in \scrL(\bT) : \\ \{A,B\} \in E(G)}} F_{AB}^2 \leq 1/m,
   \]
   completing the proof.
\end{proof}
Say that a tiling $\bT$ is {\em non-degenerate} if $\cL(\bT) \cap \cR(\bT) = \emptyset$, i.e., if no tile $A \in \bT$
touches both the left and right edges of $\region{\bT}$.
Let $\Delta_{\bT} \seteq \max_{A \in \bT} \deg_{G(\bT)}(A)$. 
If $\bS$ and $\bT$ are non-degenerate, we have the simple inequalities $\rho(\bT) \geq 1/(\Delta_{\bT} \cdot |\scrL(\bT)|) $ and $\rho(\bT) \geq 1/(\Delta_{\bT} \cdot |\scrR(\bT)|)$.
Together with \pref{lem:degree_bound}, this
yields a fact that we will employ later.

\begin{corollary}\label{cor:side}
   For any two non-degenerate tilings $\bS,\bT$ satisfying the assumptions of \pref{lem:side_product_reff},
   it holds that
   \begin{align*}
	   	\rho(\bS \tilecat \bT) &\leq \rho(\bS) + \rho(\bT) + 8 \min(\alpha_S \rho(\bS),\alpha_T \rho(\bT)) \\
	   &\lesssim_{\alpha_S, \alpha_T} \rho(\bS) + \rho(\bT).
   \end{align*}
\end{corollary}

\begin{lemma}
\label{lem:left_right_reff}
   For every $n \geq 1$, it holds that
   \[
   \rho(\bH^n) \lesssim (5/6)^n.
   \]
\end{lemma}

\begin{proof}
   Fix $n \geq 2$.
   Recalling \pref{fig:dual},
   let us consider $\bH^n$ as consisting of three (identical) tilings stacked vertically, and
   where each of these three tilings is written as $\bH^{n-1} \tilecat \bS \tilecat \bH^{n-1}$ where $\bS$ consists of two copies
   of $\bH^{n-1}$ stacked vertically.  Applying \pref{lem:side_product_reff} to $\bH^{n-1} \tilecat \bS \tilecat \bH^{n-1}$
   gives
   \begin{align*}
      \rho(\bH^n) &\leq (1/3)^2 \cdot 3 \left(2 \rho(\bH^{n-1}) + \rho(\bS) + \frac{1}{\max(|\scrR(\bH^{n-1})|,|\scrL(\bS)|)} + \frac{1}{\max(|\scrL(\bH^{n-1})|,|\scrR(\bS)|)}\right) \\
                  &\leq (1/3)^2 \cdot 3 \left(2 \rho(\bH^{n-1}) + (1/2)^2 \cdot 2 \rho(\bH^{n-1}) + \frac{2}{2 \cdot 3^{n-1}}\right) \\
      &= (5/6) \rho(\bH^{n-1}) + 3^{-n},
   \end{align*}
   where in the second inequality we have employed \pref{obs:LR}.
   This yields the desired result by induction on $n$.
\end{proof}

\begin{corollary}
   \label{cor:transient}
   The graphs $\cH_n = G(\bH^0 \tilecat \bH^1 \tilecat \cdots \tilecat \bH^n)$ satisfy
   \begin{equation}\label{eq:bddreff}
      \sup_{n \geq 1} \rho(\cH_n) < \infty.
   \end{equation}
   Hence $\cH_{\infty}$ is transient.
\end{corollary}

\begin{proof}
   Employing \pref{lem:side_product_reff}, \pref{obs:LR}, and \pref{lem:left_right_reff} together yields
   \[
      \rho(\cH_n) \lesssim \sum_{j=1}^n \left[(5/6)^j + 3^{-j}\right],
   \]
   verifying \eqref{eq:bddreff}.
   Now \pref{thm:bdd-energy} yields the transience of $\cH_{\infty}$.
\end{proof}

\section{Generalizations and unimodular constructions}
\label{sec:generalizations}

Consider a sequence $\gamma = \llangle \gamma_1,\ldots,\gamma_b\rrangle$ with $\gamma_i \in \N$.
Define a tiling $\bT_{\gamma} \in \cT$ as follows:  The unit square is partitioned into
$b$ columns of width $1/b$, and for $i \in \{1,2,\ldots,b\}$, the $i$th column has $\gamma_i$ rectangles
of height $1/\gamma_i$.
For instance, the tiling $\bH$ from \pref{fig:tiling} can be written $\bH = \bT_{\langle 3,6,3\rangle}$.

We will assume throughout this section that $\min(\gamma)=b$ and $\gamma_1=\gamma_b$.
Let us use the notation $|\gamma| \seteq \gamma_1 + \cdots + \gamma_b$.
The proof of the next lemma follows just as for \pref{lem:example_diam} using $\min(\gamma)=b$
so that there is a column in $\bT_{\gamma}^n$ of height $b^n$.

\begin{lemma}
\label{lem:HabnDiam}
For $n \geq 0$, it holds that $|\bT_{\gamma}^n| = |\gamma|^n$, and $b^n \leq \diam(\bT_{\gamma}^n) \leq 3 b^n$.
\end{lemma}

Clearly we have $\alpha_{\bT_{\gamma}} \leq |\gamma|/b$. The following lemma can be shown using a similar argument to that of \pref{lem:example_Delta_bound}. 
Note that the only symmetry required in the proof of \pref{lem:example_Delta_bound} is that the first and last column of $\bT_{\gamma}$ 
have the same geometry, and this is true since $\gamma_1=\gamma_b$.

\begin{lemma}
   \label{lem:alphaHabn}
   For any $n \geq 1$, it holds that $\alpha_{\bT_{\gamma}^n} \leq \alpha_{\bT_{\gamma}} \leq |\gamma|/b$. 
\end{lemma}
The next lemma also follows from \pref{lem:HabnDiam} and the same reasoning used in the proof of \pref{lem:growthHn}. The dependence of the implicit constant on $|\gamma|/b$ comes 
from \pref{lem:alphaHabn}.

\begin{lemma}
   \label{lem:growthHabn}
   For any $n \geq 0$, it holds that
   \[
      \left|B_{G}(A,r)\right| \asymp_{|\gamma|/b} r^{\log_b(|\gamma|)}
      \qquad \forall A \in \bT_{\gamma}^n,
      \ 1 \leq r \leq \diam(G(\bT_{\gamma}^n)).
   \]
\end{lemma}

\subsection{Degrees of growth}
\label{sec:deg-growth}

Consider $b, k \in \N$ with $k \geq b \geq 4$, and define the sequence
\[
   \gamma^{(b,k)} \seteq \llangle b,
   \underbrace{\left\lceil \frac{k-3}{b-3} \right\rceil b, \ldots, \left\lceil \frac{k-3}{b-3} \right\rceil b}_{(k-3) \bmod (b-3) \textrm{ copies}}, b,
   \underbrace{\left\lfloor\frac{k-3}{b-3} \right\rfloor b, \ldots, \left\lfloor \frac{k-3}{b-3} \right\rfloor b}_{(b-3)-[(k-3) \bmod (b-3)] \textrm{ copies}}, b\rrangle.
\] 
Denote $\bT_{(b,k)} \seteq \bT_{\gamma^{(b,k)}}$ and note that $|\gamma^{(b,k)}|=bk$.
Define $\mathsf{d}_g(b,k) \seteq \log_b(bk)$,
and $\Gamma_{b,k} \seteq \sum_{i=1}^b 1/\gamma_i^{(b,k)}$.
\begin{observation}\label{obs:column}
   The following facts hold for $k \geq b \geq 4$ and $n \geq 0$:
   \begin{enumerate}
      \item[(a)] There are $b^n$ tiles in the left- and right-most columns of $\bT^n_{(b,k)}$.
      \item[(b)] If a pair of consecutive columns in $\bT^n_{(b,k)}$ have heights
   $h$ and $h'$, then $\min(h,h')$ divides $\max(h,h')$.
   \end{enumerate}
\end{observation}

Now observe that \pref{lem:growthHabn} yields the following.

\begin{corollary}
\label{cor:sequence_log_growth}
The family of graphs $\cF = \big\{G\big(\bT_{(b,k)}^n\big) : n \geq 0\big\}$ has uniform polynomial
growth of degree $\mathsf{d}_g(b,k)$ in the sense that
\[
   |B_G(x,r)| \asymp_k r^{\mathsf{d}_g(b,k)},\qquad \forall G \in \cF, x \in V(G), 1 \leq r \leq \diam(G)\,.
\]
For any rational $p/q \geq 2$, one can achieve $\mathsf{d}_g(b,k)=p/q$ by taking $b=4^q$ and $k=4^{p-q}$.
\end{corollary}

\begin{remark}[Arbitrary real degrees $d > 2$]
   \label{rem:real-degrees}
   We note that by considering more general products of tilings, one can obtain planar graphs of uniform polynomial growth of any real degree $d > 2$ for which the main results of this paper still hold. Instead of working with the family of powers $\{\bT_\gamma^n\}$ for a fixed tiling $\bT_\gamma$, one defines an infinite sequence $\langle \gamma^{(1)}, \gamma^{(2)}, \ldots \rangle$, and examines the family of graphs $\bT_{\gamma^{(n)}} \circ \cdots \circ \bT_{\gamma^{(1)}}$.

   More concretely, 
   fix some real $d > 2$, and let us consider a sequence $\{h_n : n \geq 1\}$ of nonnegative integers.
   Also define $\gamma^{(n)} \seteq \gamma^{(4,4+h_n)}$ and $\bT^{(n)} \seteq \bT_{\gamma^{(n)}} \circ \cdots \circ \bT_{\gamma^{(1)}}$.
Then $|\bT^{(n)}| = 4^n \prod_{j=1}^n (4+h_j)$, and $\diam(\bT^{(n)}) \asymp 4^n$.
By a similar argument as in \pref{lem:growthHn} based on the recursive structure, it holds that for $i \in \{1,2,\ldots,n\}$,
balls of radius $\asymp 4^i$ in $\bT^{(n)}$ have volume $\asymp_K 4^i \prod_{j=1}^i (4+h_j)$, where $K \seteq \max \{ h_n : n \geq 1\}$.
Given our choice of $\{h_1,\ldots,h_{n-1}\}$, we
choose $h_n \geq 0$ as large as possible subject to
\[
   \sum_{j=1}^n \log_4(1+h_j/4) \leq (d-2) n.
\]
It is straightforward to argue that $K \lesssim_d 1$, and
\[
   \left|\sum_{j=1}^n \log_4(1+h_j/4) - (d-2) n\right| \lesssim_d 1,
\]
implying that $4^n \prod_{j=1}^n (4+h_j) \asymp_d 4^{dn}$ for every $n \geq 1$.
It follows that the graphs $\{\bT^{(n)} : n \geq 1 \}$ have uniform polynomial
growth of degree $d$.
\end{remark}

Let us now return to the graphs $\bT^n_{(b,k)}$ and analyze the effective resistance across them.

\begin{lemma}
   \label{lem:eff-resist-general}
   For every $n \geq 1$, it holds that
   \[
      \Gamma^n_{b,k} \lesssim_k \rho\left(\bT^n_{(b,k)}\right) \lesssim \Gamma^n_{b,k}\,.
   \]
\end{lemma}

\begin{proof}
   Fix $n \geq 2$ and write $\bT^n_{(b,k)} = \bT_{(b,k)} \tileprod \bT^{n-1}_{(b,k)}$ as
   $\bA_1 \tilecat \bA_2 \tilecat \cdots \tilecat \bA_b$ where, for $1 \leq i \leq b$, 
   each $\bA_i$ is a vertical stack of $\gamma_i^{(b,k)}$ copies of $\bT^{n-1}_{(b,k)}$.
   Since $\rho(\bA_i) = \rho(\bT_{(b,k)}^{n-1})/\gamma_i^{(b,k)}$ by the parallel law for
   effective resistances, applying \pref{lem:side_product_reff} to $\bA_1 \tilecat \bA_2 \tilecat \cdots \tilecat \bA_b$ gives
   \begin{align*}
      \rho\left(\bT^n_{(b,k)}\right) &\leq \sum_{i-1}^b \rho(\bT_{(b,k)}^{n-1})/\gamma_i^{(b,k)} +
      \sum_{i=1}^{b-1} \frac{1}{\min\left(\scrR(\bA_i),\scrL(\bA_{i+1})\right)} 
                                     \leq \rho(\bT^{n-1}_{(b,k)}) \Gamma_{b,k} + b^{1-n},
   \end{align*}
   where in the second inequality we have employed $\min\left(\scrR(\bA_i),\scrL(\bA_{i+1})\right) \geq b^{n}$
   which follows from \pref{obs:column}(a).
   Finally, observe that $\Gamma_{b,k} \geq 1/\gamma_1^{(b,k)} + 1/\gamma_b^{(b,k)} = 2/b$, and therefore
   the desired upper bound follows by induction.

   For the lower bound, note that since the degrees in $\bT^n_{(b,k)}$ are bounded by $k$, the Nash-Williams inequality (see,
   e.g., \cite[\S 5]{lp:book}) gives
   \begin{equation}\label{eq:nash-williams}
      \rho(\bT^n_{(b,k)}) \gtrsim_k \sum_{i=2}^{b^n-1} \frac{1}{|K_i|} \gtrsim \sum_{i=1}^{b^n} \frac{1}{|K_i|} = \Gamma_{b,k}^n,
   \end{equation}
   where $K_i$ is the $i$th column of rectangles in $\bT^n_{(b,k)}$, and the last equality
   follows by a simple induction.
\end{proof}

The next result establishes \pref{thm:transient-intro}. 

\begin{theorem}
   \label{thm:eff-resist-general}
   For every $k > b$,
   the graphs $\cT_n^{(b,k)} \seteq G\left(\bT^0_{(b,k)} \tilecat \bT^1_{(b,k)} \tilecat \cdots \tilecat \bT^n_{(b,k)}\right)$
   satisfy
   \begin{equation}\label{eq:cor-gen-resist}
      \sup_{n \geq 1} \rho\left(\cT_n^{(b,k)}\right) < \infty.
   \end{equation}
   Hence the limit graph $\cT_{\infty}^{(b,k)}$ is transient.
   Moreover, $\cT_{\infty}^{(b,k)}$ has uniform polynomial growth of degree $\mathsf{d}_g(b,k)$.
\end{theorem}

\begin{proof}
   Employing \pref{lem:side_product_reff}, \pref{obs:column}(a),
     and \pref{lem:eff-resist-general} together yields
     \[
      \rho\left(\cT_n^{(b,k)}\right) \lesssim \sum_{j=1}^n \left(\Gamma_{b,k}^j + b^{-j}\right).
     \]
     For $k > b$, we have $\max(\gamma^{(b,k)}) > b$ and $\min(\gamma^{(b,k)})=b$, hence $\Gamma_{b,k} < 1$,
     verifying \eqref{eq:cor-gen-resist}.
     Now \pref{thm:bdd-energy} yields transience of $\cT_{\infty}^{(b,k)}$.

     Uniform polynomial growth of degree $\mathsf{d}_g(b,k)$
     follows from \pref{cor:sequence_log_growth} as in the proof of \pref{lem:Hn-uniform-growth}.
\end{proof}

\subsection{The distributional limit}
\label{sec:dl}

Fix $k \geq b \geq 4$ and take $G_n \seteq G(\bT_{(b,k)}^n)$.
Since the degrees in $\{G_n\}$ are uniformly bounded, the sequence
has a subsequential distributional limit, and in all arguments
that follow, we could consider any such limit.
But let us now argue that if $\mu_n$ is the law of $(G_n,\rho_n)$
with $\rho_n \in V(G_n)$ chosen according to the stationary measure,
then the measures $\{\mu_n : n \geq 0\}$ have a distributional limit.

\begin{lemma}\label{lem:gbk}
   For any $k \geq b \geq 4$, there is a reversible random graph $(G_{b,k},\rho)$ such that
   $\{(G_n,\rho_n)\} \todl (G_{b,k},\rho)$.
   Moreover, almost surely $G_{b,k}$ has uniform polynomial volume growth of degree $\mathsf{d}_g(b,k)$.
\end{lemma}

\begin{proof}
   It suffices to prove that $\{(G_n,\rho_n)\}$ has a limit $(G_{b,k},\rho)$.
   Reversibility of the limit then follows automatically
   (as noted in \pref{sec:graph-limits}), and the degree of growth
   is an immediate consequence of \pref{cor:sequence_log_growth}.
   It will be slightly easier to show that the sequence $\{(G_n,\hat{\rho}_n)\}$ has
   a distributional limit, with $\hat{\rho}_n \in V(G_n)$ chosen uniformly at random. As noted in \pref{sec:graph-limits},
   the claim then follows from \cite[Prop. 2.5]{bc12} (the correspondence between unimodular and reversible random graphs
   under degree-biasing).

   Let $\mu_{n,r}$ be the law of $B_{G_n}(\hat{\rho}_n,r)$.  It suffices to show that
   the measures $\{\mu_{n,r} : n \geq 0\}$ converge for every fixed $r \geq 1$, and then a standard application
   of Kolmogorov's extension theorem proves the existence of a limit.

   For a tiling $\bT$ of a rectangle $R$, let $\partial \bT$ denote the set of tiles
   that intersect some side of $R$.
   Define the neighborhood
   $N_r(\partial \bT_{(b,k)}^n) \seteq \{ v \in \bT^n_{(b,k)} : d_{G_n}(v, \partial \bT_{(b,k)}^n) \leq r\}$
   and abbreviate $\mathsf{d} = \mathsf{d}_g(b,k)$.
   Then $|\partial \bT^n_{(b,k)}| \leq 4b^n$, so \pref{cor:sequence_log_growth} gives
   \[
      \left|N_r(\partial \bT^n_{(b,k)})\right| \lesssim_k b^n r^{\mathsf{d}}\,.
   \]
   Since $|\bT^n_{(b,k)}| = (bk)^n$, it follows that
   \[
      1-\Pr\left[\cE_{r,n}\right] \lesssim_k k^{-n} r^\mathsf{d},
   \]
   where $\cE_{r,n}$ is the event $\{ B_{G_{n}}(\hat{\rho}_{n}, r) \cap \partial \bT^{n}_{(b,k)} = \emptyset \}$.

   Now write $\bT^n_{(b,k)} = \bT_{(b,k)} \circ \bT^{n-1}_{(b,k)}$, and note that $\hat{\rho}_n$ falls into one
   of the $|\gamma^{(b,k)}|=bk$ copies of $G_{n-1}$ and is, moreover, uniformly distributed in that copy.
   Therefore we can naturally couple $(G_n,\hat{\rho}_n)$ and $(G_{n-1},\hat{\rho}_{n-1})$ by
   identifying $\hat{\rho}_n$ with $\hat{\rho}_{n-1}$.
   Moreover, conditioned on the event $\cE_{r,n-1}$,
   we can similarly couple $B_{G_n}(\hat{\rho}_n,r)$ and $B_{G_{n-1}}(\hat{\rho}_{n-1},r)$.

   It follows that, for every $r \geq 1$,
   \[
      d_{TV}\left(\mu_{n-1,r},\mu_{n,r}\right) \leq 1-\Pr[\cE_{r,n-1}] \lesssim_k k^{-n} r^{\mathsf{d}}\,.
   \]
   As the latter sequence is summable, it follows that $\{\mu_{n,r}\}$ converges for every fixed $r \geq 1$,
   completing the proof.
\end{proof}

\subsection{Speed of the random walk}
\label{sec:rw-speed}

Let $\{X_t\}$ denote the random walk on $G_{b,k}$ with $X_0 = \rho$.
Our first goal will be to prove a lower bound on the speed of the walk.
Define:
\begin{align*}
   \mathsf{d}_w(b,k) &\seteq \mathsf{d}_g(b,k) + \log_b(\Gamma_{b,k}).
\end{align*}
We will show that $\mathsf{d}_w(b,k)$ is related to the speed exponent
for the random walk.

\begin{theorem}\label{thm:Gbk-speed}
   Consider any $k \geq b \geq 4$.
   It holds that for all $T \geq 1$,
   \begin{equation}\label{eq:Gbk-speed}
      \E\left[d_{G_{b,k}}\!\left(X_T,X_0\right) \mid X_0 = \rho\right] \gtrsim_k T^{1/\mathsf{d}_w(b,k)}.
   \end{equation}
\end{theorem}

Before proving the theorem, let us observe that it yields \pref{thm:speed-intro}.
Fix $k \geq b \geq 4$.
Observe that for any positive integer $p \geq 1$, we have
$\mathsf{d}_g(b^p,k^p)=\mathsf{d}_g(b,k)$.
On the other hand,
\begin{align}
   \mathsf{d}_g(b^p,k^p) - \mathsf{d}_w(b^p,k^p) &= -\log_{b^p}(\Gamma_{b^p,k^p}) \nonumber \\
                               &\geq - \log_{b^p}\left(3 b^{-p} + (b/k)^p\right) - o_p(1) \nonumber \\
                               &\geq \min\left(1, \log_b(k)-1\right) - o_p(1) \nonumber \\
                               &\geq \min\left(1, \mathsf{d}_g(b^p,k^p)-2\right) - o_p(1)\,. \label{eq:dw-calc}
\end{align}
So for every $\e > 0$, there is some $p=p(\e)$ such that
\[
   \mathsf{d}_w(b^p,k^p) \leq \max\left(2,\mathsf{d}_g(b^p,k^p)-1\right) + \e\,,
\]
and moreover $G_{b^p,k^p}$ almost surely has uniform polynomial growth of degree $\mathsf{d}_g(b,k)$.
Combining this with the construction of \pref{cor:sequence_log_growth} for all rational
$d \geq 2$ yields \pref{thm:speed-intro}.

\subsubsection{The linearized graphs}
\label{sec:linearized}

Fix integers $k \geq b \geq 4$ and $n \geq 1$,
and let us consider now the (weighted) graph $L=L_{(b,k)}^n$
derived from $G=G(\bT_{(b,k)}^n)$ by
identifying every column of rectangles into a single vertex.
Thus $|V(L)| = b^n$.

We connect two vertices $u,v \in V(L)$ if their corresponding columns $\cC_u$ and $\cC_v$ in $G$
are adjacent, and we define the conductances $c_{uv} \seteq |E_G(\cC_u,\cC_v)|$,
where $E_G(S,T)$ denotes the number of edges between two subsets $S,T \subseteq V(G)$.
Define additionally $c_{uu} \seteq 2 |\cC_u|$ and
\begin{align*}
   c_u &\seteq c_{uu} + \sum_{v : \{u,v\} \in E(L)} c(\{u,v\})\,.
\end{align*}

Let us order the vertices of $L$ from left to right as $V(L) = \{\ell_1, \ldots, \ell_{b^n}\}$.
The series law for effective resistances gives the following.
\begin{observation}
\label{obs:linear-reff}
	For $1 \leq i \leq t \leq j\leq b^n$, we have 
   \[
      \reff^L(\ell_i \leftrightarrow \ell_j) = \reff^L(\ell_i \leftrightarrow \ell_t) + \reff^L(\ell_t \leftrightarrow \ell_j)
   \]
\end{observation}

We will use this to bound the resistance between any pair of columns.

\begin{lemma}
\label{lem:reff-times-conductance}
	If $1 \leq s < t \leq b^n$, then
   \begin{equation}\label{eq:reff-times-goal}
   \reff^L(\ell_s \leftrightarrow \ell_t) \cdot \left(c_{\ell_s} + c_{\ell_{s+1}} + \cdots + c_{\ell_t}\right) \asymp_k (\Gamma_{b,k} \cdot bk)^{\log_b(t-s)}\,.
   \end{equation}
\end{lemma} 
\begin{proof}
Let us first establish the upper bound.
Denote $h \seteq \lceil \log_b(t-s) \rceil$, $\bT \seteq \bT_{(b,k)}$, and $\Gamma \seteq \Gamma_{b,k}$.
Write $\bT^n = \bT^{n-h} \tileprod \bT^h$ and
along this decomposition,
partition $\bT^n$ into $b^{n-h}$ sets of tiles $\cD_1, \ldots, \cD_{b^{n-h}}$, where
each $\cD_i$ is formed from adjacent columns
\begin{equation}\label{eq:tilepart}
   \cD_i \seteq \cC_{(i-1)\cdot b^h + 1} \cup \cdots \cup \cC_{i\cdot b^h}\,.
\end{equation}
Suppose that, for $1 \leq i \leq b^{n-h}$, the tiling $\bT^{n-h}$ has $\beta_i$ tiles
in its $i$th column.
Then $\cD_i$ consists of $\beta_i$ copies of $\bT^h$ stacked atop each other.

Thus we have $|\cD_i| = \beta_i |\bT^h|$, and furthermore $\rho(\cD_i) \leq \rho(\bT^h)/\beta_i$,
hence
\begin{equation}
\label{eq:gammatimesa}
\rho(\cD_i) \cdot |\cD_i| \leq \rho(\bT^h) \cdot |\bT^h| \lesssim \Gamma^h \cdot (bk)^h,
\end{equation}
where the last inequality uses \pref{lem:eff-resist-general}.

Let $1 \leq i \leq j \leq b^{n-h}$ be such that $\cC_s \subseteq \cD_i$ and $\cC_t \subseteq \cD_j$.
Since $t \leq s + b^h$, and each set $\cD_i$ consists of $b^h$ consecutive columns,
it must be that $j \leq i+1$.
If $i=j$, then $|\cC_s|+\cdots+|\cC_t| \leq |\cD_i|$, and
\pref{obs:linear-reff} gives
\[
   \reff^L(\ell_s \leftrightarrow \ell_t) \leq \rho(\cD_i),
\]
thus \eqref{eq:gammatimesa} yields \eqref{eq:reff-times-goal}, as desired.

Suppose, instead, that $j=i+1$.
From \pref{lem:alphaHabn}, we have $\alpha_{\bT^{n-h}} \leq |\gamma|/b \leq k$. 
Therefore $1/k \leq \beta_{i+1}/\beta_i \leq k$.
Since the degrees in $G(\bT^n_{(b,k)})$ are bounded by $k$,
this yields the following claim, which we will also employ later.
\begin{claim}\label{claim:Lmass}
   For any $\ell_i \in V(L)$, we have
   \begin{equation}\label{eq:Lmass1}
      c_{\ell_i} \asymp_k |\cC_i|,
   \end{equation}
   and
   for any $D > 0$ and columns $\cC_a \in \cD_i$ and $\cC_b \in \cD_j$ with $|i-j| \leq D$,
   it holds that
   \begin{equation}\label{eq:Lmass2}
       c_{\ell_a}  \asymp_k   |\cC_a| \asymp_{k,D} |\cC_b| \asymp_k c_{\ell_b}.
   \end{equation}
\end{claim}

Thus using \pref{cor:side} (and noting that each $\cD_i$ is non-degenerate) along with \eqref{eq:gammatimesa} gives
\[
   \rho(\cD_i \cup \cD_{i+1}) = \rho(\cD_i \tilecat \cD_{i+1}) \lesssim_k  \rho(\cD_i) + \rho(\cD_{i+1}) \lesssim_k \rho(\bT^h)/\beta_i.
\]
Since it also holds that $|\cD_i|+|\cD_{i+1}| = (\beta_i+\beta_{i+1}) |\bT^h| \leq 2k \beta_i |\bT^h|$, \pref{obs:linear-reff} gives
\[
   \reff^L(\ell_s \leftrightarrow \ell_t) \cdot \left(|\cC_s|+\cdots+|\cC_t|\right) \leq \rho(\cD_i \cup \cD_{i+1}) |\cD_i \cup \cD_{i+1}| \lesssim_k \rho(\bT^h) |\bT^h|,
\]
and again \eqref{eq:gammatimesa} establishes \eqref{eq:reff-times-goal}.
Now \eqref{eq:Lmass1}
completes the proof of the upper bound.

For the lower bound, define $h' \seteq \lfloor \log_b (t-s)\rfloor - 1$ and decompose $\bT^n = \bT^{n-h'} \tileprod \bT^{h'}$.
Partition
$\bT^n$ similarly into $b^{n-h'}$ sets of tiles $\cD_1,\ldots,\cD_{b^{n-h'}}$.
Suppose that $\cC_s \subseteq \cD_i$ and $\cC_t \subseteq \cD_j$, and note that
the width of each $\cD_i$ is $b^{h'}$ and $b^{h'+2} \geq t-s \geq b^{h'+1}$, hence
$j > i+1$.
Therefore using again \pref{obs:linear-reff} and the Nash-Williams inequality, we have
\[
   \reff^L(\ell_s \leftrightarrow \ell_t) \geq \sum_{j=s+1}^{t-1} \frac{1}{c_{\ell_j}} \stackrel{\eqref{eq:Lmass1}}{\gtrsim_k} \sum_{j=s+1}^{t-1} \frac{1}{|\cC_j|}
   \geq \sum_{j=(i-1) b^h+1}^{i b^h} \frac{1}{|\cC_j|}
   = \frac{1}{\beta_{i+1}} \Gamma_{b,k}^n,
\]
where the final inequality uses \eqref{eq:nash-williams}.
Note also that
\[
   |\cC_s|+\cdots+|\cC_t| \geq |\cD_{i+1}| = \beta_{i+1} |\bT^{h'}| = \beta_{i+1} (bk)^{h'}.
\]
An application of \eqref{eq:Lmass1} completes the proof of the lower bound.
\end{proof}

\subsubsection{Rate of escape in $L$}

Consider again the linearized graph $L = L^n_{(b,k)}$ with conductances $c : E(L) \to \R_+$
defined in \pref{sec:linearized}, and let $\{Y_t\}$ be the random walk on $L$ defined by
\begin{equation}\label{eq:ywalk}
   \Pr[Y_{t+1} = v \mid Y_t = u] = \frac{c_{uv}}{c_u}, \quad \{u,v\} \in E(L) \textrm{ or } u=v.
\end{equation}
Let $\pi_L$ be the stationary measure of $\{Y_t\}$.

For a parameter $1 \leq h \leq n$, consider the decomposition $\bT^n = \bT^{n-h} \tileprod \bT^h$, and
let $V_1,V_2,\ldots,V_{b^{n-h}}$ be a partition of $V(L)$ into 
continguous subsets
with $|V_1|=|V_2|=\cdots=|V_b^{n-h}|=b^h$.

Let $\{Z_i : i \in \{1,2,\ldots,b^{n-h}\}\}$ be a collection of independent random variables with
   \[
      \Pr[Z_i = v] = \frac{\pi_L(v)}{\pi_L(V_i)},\quad v \in V_i.
   \]
Define the random time $\tau(h)$ as follows:  Given $Y_0 \in V_j$, let $\tau(h)$ be the first time $\tau \geq 1$ at which
\begin{align*}
   Y_{\tau} \in \{Z_{j-2},Z_{j+2}\} \qquad & 3 \leq j \leq b^{n-h}-2 \\
   Y_{\tau} = Z_{j+2} \qquad & j \in \{1,2\} \\
   Y_{\tau} = Z_{j-2} \qquad & j \in \{b^{n-h}-1,b^{n-h}\}.
\end{align*}
The next lemma shows that the law of the walk stopped at time $\tau(h)$ is
within a constant factor of the stationary measure.

\begin{lemma}
   \label{lem:almost-pi}
   Suppose $Y_0$ is chosen according to $\pi_L$.
   Then for every $v \in V(L)$,
   \[
      \Pr\left[Y_{\tau(h)} = v\right] \asymp_k \pi_L(v).
   \]
\end{lemma} 

\begin{proof}
   Consider some $5 \leq j \leq b^{n-h}-4$ and $v \in V_j$.
The proof for the other cases is similar.  Let $\cE$ denote the event $\{Y_0 \in \{V_{j-2},V_{j+2}\}\}$.
   The conditional measure is
   \[
      \Pr[Y_0=u \mid \cE] = \frac{\pi_L(u)}{\pi_L(V_{j-2})+\pi_L(V_{j+2})},\quad u \in V_{j-2} \cup V_{j+2}.
   \]
   Consider three linearly ordered vertices $v,u,w \in V(L)$, i.e., such that $v,w$ are in distinct
   connected components of $L[V(L) \setminus \{u\}]$).
   Let $p_{u}^{v \prec w}$ denote the probability that the random walk,
   started from $Y_0 = u$ hits $v$ before it hits $w$.
   Now we have:
   \begin{equation}
      \Pr\left[Y_{\tau(h)} = v\right] 
                                      = \frac{\pi_L(v)}{\pi_L(V_j)} \left(
         \sum_{u \in V_{j-2}} \pi_L(u) \sum_{w \in V_{j-4}} \frac{\pi_L(w)}{\pi_L(V_{j-4})} p_u^{v \prec w} 
         +
      \sum_{u \in V_{j+2}} \pi_L(u) \sum_{w \in V_{j+4}} \frac{\pi_L(w)}{\pi_L(V_{j+4})} p_u^{v \prec w}\right) \label{eq:conprob}
   \end{equation}
   It is a classical fact (see \cite[Ch. 2]{lp:book}) that
   \[
      p_u^{v \prec w} = \frac{\reff^L(u \lra w)}{\reff^L(u \lra v) + \reff^L(u \lra w)} =
      \frac{\reff^L(u\lra w)}{\reff^L(v \lra w)},
   \]
   where the final equality uses \pref{obs:linear-reff} and the fact that $v,u,w$ are linearly ordered.

   Thus from 
   \pref{lem:reff-times-conductance}
   and \eqref{eq:Lmass2},
   whenever $w \in V_{j-4}, u \in V_{j-2}, v \in V_{j}$ or $u \in V_{j+2}, v \in V_j, w \in V_{j+4}$, it holds that
   \[
      p_{u}^{v \prec w} \asymp_k 1 
   \]
   Another application of \eqref{eq:Lmass2} gives
   \[
      \pi_L(V_{j-2})\asymp_k \pi_L(V_{j-4}) \asymp_k \pi_L(V_j) \asymp_k \pi_L(V_{j+2}),
   \]
   hence \eqref{eq:conprob} gives
   \[
      \Pr\left[Y_{\tau(h)} = v\right] \asymp_k \pi_L(v),
   \]
   completing the proof.
\end{proof}

\begin{lemma}\label{lem:hitting-time}
   It holds that $\E[\tau(h) \mid Y_0] \lesssim_k b^{h\,\mathsf{d}_w(b,k)}$.
\end{lemma}

\begin{proof}
   Consider a triple of vertices $u \in V_{j}, v \in V_{j-2}, w\in V_{j+2}$ for $3 \leq j \leq b^{n-h} - 2$. 
   Let $\tau_{v,w}$ be the smallest time $\tau \geq 0$ such that $X_\tau \in \{v,w\}$, and denote
   \[
      t_{u}^{v,w} \seteq \E[\tau_{v,w} \mid Y_0 = u].
   \]
   Then the standard connection between hitting times and effective resistances \cite{CRRST96}
   yields 
   \begin{align*}
      t_u^{v,w} \leq 2 \left(\sum_{i=j-2}^{j+2} \sum_{x \in V_i} c_x\right) \min\left(\reff^L(u \leftrightarrow v), \reff^L(u \leftrightarrow w)\right)
                \lesssim_k \left(\Gamma_{b,k} bk\right)^{h},
   \end{align*}
   where the last line employs \pref{lem:reff-times-conductance}.
   Recalling that $\mathsf{d}_w(b,k) = \log_b(bk \Gamma_{b,k})$, this yields
   \[
      \E[\tau(h) \mid Y_0 = v] \lesssim_k b^{h\,\mathsf{d}_w(b,k)}
   \]
   for any $v \in V_3 \cup V_4 \cup \cdots \cup V_{b^{n-h}-2}$.
   A one-sided variant of the argument follows in the same manner for $v \in V_j$ when $j \leq 2$ or $j \geq b^{n-h} - 1$.
\end{proof}

\begin{lemma}
   \label{lem:Lspeed}
   Let $Y_0$ have law $\pi_L$.
   There is a number $c_k > 0$ such that
   for any $T \leq c_k \diam(L)^{\mathsf{d}_w(b,k)}$, we have
   \[
      \E\left[d_L(Y_0,Y_T)\right] \gtrsim_k T^{1/\mathsf{d}_w(b,k)}.
   \]
\end{lemma}

\begin{proof}
   First, we claim that for every $T \geq 1$,
   \begin{equation}\label{eq:anytime}
      \E\left[d_L(Y_0,Y_T)\right] \geq \frac12 \max_{0 \leq t \leq T} \E\left[d_L(Y_0,Y_t) - d_L(Y_0,Y_1)\right].
   \end{equation}
   Let $s' \leq T$ be such that
   \[
      \E[d_L(Y_0,Y_{s'})] = \max_{0 \leq t \leq T} \E[d_L(Y_0,Y_t)].
   \]
   Then there exists an even time $s \in \{s',s'-1\}$ such that $\E[d_L(Y_0,Y_s)] \geq \E[d_L(Y_0,Y_{s'})-d_L(Y_0,Y_1)]$.
   Consider $\{Y_t\}$ and an identically distributed walk $\{\tilde{Y}_t\}$ such that
   $\tilde{Y}_t=Y_t$ for $t \leq s/2$ and $\tilde{Y}_t$ evolves independently after time $s/2$.
   By the triangle inequality, we have
   \[
      d_L(Y_0,\tilde{Y}_T) + d_L(\tilde{Y}_T,Y_s) \geq d_L(Y_0, Y_s).
   \]
   But since $\{Y_t\}$ is stationary and reversible, $(Y_0,\tilde{Y}_T)$ and $(\tilde{Y}_T, Y_s)$ have the same law
   as $(Y_0,Y_T)$.  Taking expectations yields \eqref{eq:anytime}.

   Let $h \in \{1,2,\ldots,n\}$ be the largest value such that $\E[\tau(h)] \leq T$.
   We may assume that $T$ is sufficiently large so that $\E[\tau(1)] \leq T$, and
   \pref{lem:hitting-time} guarantees that 
   \begin{equation}\label{eq:dist-large}
      b^h \gtrsim_k T^{1/\mathsf{d}_w(b,k)},
   \end{equation}
   as long as $h < n$ (which gives our restriction $T \leq c_k b^{h \mathsf{d}_w(b,k)} = c_k \diam(L)^{1/\mathsf{d}_w(b,k)}$ for some
   $c_k > 0$).

   From the definition of $\tau(h)$, we have
   \[
      d_L(Y_0, Y_{\tau(h)}) \geq b^h,
   \]
   hence the triangle inequality implies
   \begin{equation}\label{eq:lb1}
      d_L(Y_0,Y_{2T}) \geq \1_{\{\tau(h) \leq 2T\}} \left(b^h - d_L(Y_{\tau(h)}, Y_{2T})\right).
   \end{equation}

   Again, let $\{\tilde{Y}_t\}$ be an independent copy of $\{Y_t\}$.
   Then since $\Pr(\tau(h) \leq 2T) \geq 1/2$, \pref{lem:almost-pi} implies
   \begin{equation*}\label{eq:y-compare}
      \Pr\left[Y_{\tau(h)}=v \mid \{\tau(h) \leq 2T\}\right] \leq 2 \Pr\left[Y_{\tau(h)}=v\right] \lesssim_k  \Pr\left[\tilde{Y}_0=v\right].
   \end{equation*}
   Therefore,
   \begin{align*}
      \E\left[\1_{\{\tau(h) \leq 2T\}} \,d_L(Y_{\tau(h)}, Y_{2T})\right] 
         &\lesssim_k \E\left[\1_{\{\tau(h) \leq 2T\}} \,d_L(\tilde{Y}_{0},\tilde{Y}_{2T-\tau(h)})\right].
   \end{align*} 
   Define $\eta \seteq b^{-h} \max_{0 \leq t \leq 2T} \E[d_L(Y_0,Y_t)]$. Using the above bound yields 
   \[
      \E\left[\1_{\{\tau(h) \leq 2T\}} \,d_L(Y_{\tau(h)}, Y_{2T})\right] \leq C(k) \Pr(\tau(h) \leq 2T)\,\eta b^h,
   \] 
   for some number $C(k)$.
   Taking expectations in \eqref{eq:lb1} gives
   \[
      \E\left[d_L(Y_0,Y_{2T})\right] \geq \Pr(\tau(h) \leq 2T) \left(1-\eta C(k)\vphantom{\bigoplus}\right) b^h
      \geq \frac12 \left(1-\eta C(k)\vphantom{\bigoplus}\right) b^h.
   \]
   If $\eta \leq 1/(2 C(k))$, then $\E[d_L(Y_0,Y_{2T})] \geq b^h/4$.
   If, on the other hand, $\eta > 1/(2 C(k))$, then
   \pref{eq:anytime} yields
   \[
      \E[d_L(Y_0,Y_{2T})] \geq \frac12 \left(\eta b^h - 1\right) \gtrsim_k b^h.
   \]
   Now \eqref{eq:dist-large} completes the proof.
\end{proof}

\subsubsection{Rate of escape in $G_{b,k}$}
\label{sec:hit-gbk}

\begin{figure}
   \centering
   \includegraphics[width=4cm]{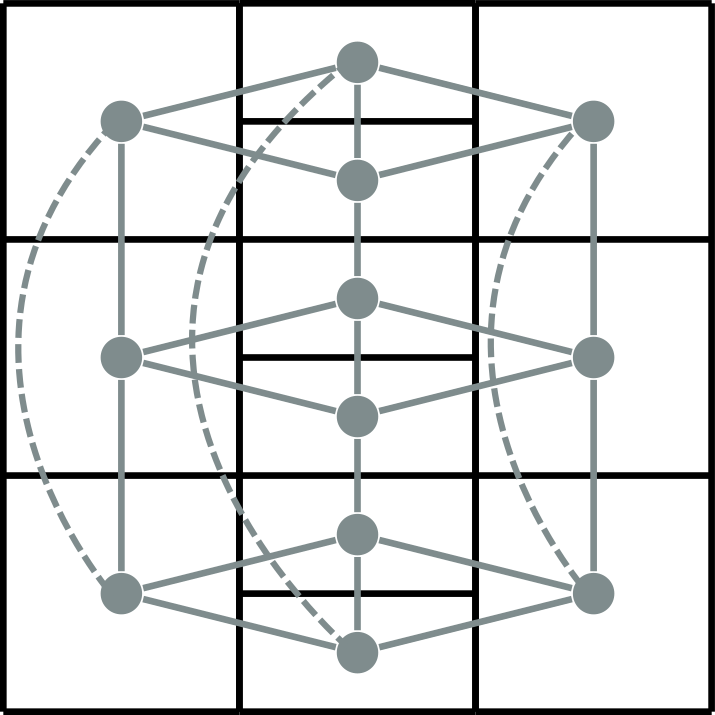}
   \caption{The cylindrical graph $\tilde{G}$ for $G=G(\bT_{\langle 3,6,3\rangle})$.  The new edges are dashed.\label{fig:torus}}
\end{figure}

Consider now the graphs $G_n \seteq G(\bT^n_{(b,k)})$ for some $k \geq b \geq 4$ and $n \geq 1$.
Let us define the cylindrical version $\tilde{G}_n$ of $G_n$
with the same vertex set, but additionally and edge from
the top tile to the bottom tile in every column (see \pref{fig:torus}).
If we choose $\tilde{\rho}_n \in V(\tilde{G}_n)$ according to the
stationary measure on $\tilde{G}_n$, then clearly $\{(\tilde{G}_n,\tilde{\rho}_n)\} \todl (G_{b,k},\rho)$ as well.

Define also $L_n \seteq L^n_{b,k}$.
Because of \pref{obs:column}(b), the graph $\tilde{G}_n$ has vertical symmetry:
Tiles within a column all have the same degree and, more specifically,
have the same number of neighbors on the left and on the right. 
Let $\pi_n : V(G_n) \to V(L_n)$ denote the projection map and observe that
\[
   d_{\tilde{G}_n}(u,v) \geq d_{L_n}\!\left(\pi_n(u),\pi_n(v)\right), \qquad \forall u,v \in V(G_n).
\]
Let $\{X^{(n)}_t\}$ denote the random walk on $\tilde{G}_n$ with $X^{(n)}_0=\tilde{\rho}_n$, and let $\{Y^{(n)}_t\}$ be the 
stationary random walk
on $L_n$ defined in \eqref{eq:ywalk}.

Note that, by construction, $\{\pi_n(X_t^{(n)})\}$ and $\{Y_t^{(n)}\}$ have the same law, and
therefore
\begin{equation}\label{eq:lb-proj}
   \E\left[d_{\tilde{G}_n}\left(X^{(n)}_0,X^{(n)}_T\right)\right] \geq \E\left[d_{L_n}\left(Y^{(n)}_0,Y^{(n)}_T\right)\right].
\end{equation}

With this in hand, we can establish speed lower bounds in the limit $(G_{b,k},\rho)$.

\begin{proof}[Proof of \pref{thm:Gbk-speed}]
   Observe that \eqref{eq:lb-proj} in conjunction with \pref{lem:Lspeed} gives, for
   every $T \leq c_k (\diam(\tilde{G}_n)/2)^{\mathsf{d}_w(b,k)}$,
   \[
      \E\left[d_{\tilde{G}_n}\!\left(X^{(n)}_0,X^{(n)}_T\right)\right] \gtrsim_k T^{1/\mathsf{d}_w(b,k)}
   \]
   Since $\{(\tilde{G}_n,\rho_n)\} \todl (G,\rho)$ by \pref{lem:gbk}, it holds that
   if $\{X_t\}$ is the random walk on $G$ with $X_0=\rho$, then for all $T \geq 1$,
   \[
      \E\left[d_{G}(X_0,X_T)\right] \gtrsim_k T^{1/\mathsf{d}_w(b,k)}.\qedhere
   \]
\end{proof}

\subsection{Annular resistances}
\label{sec:anular-resist}

We will establish \pref{thm:intro-resist} by proving the following.

\begin{theorem}\label{thm:annular-gbk}
   For any $k \geq b \geq 4$, there is a constant $C=C(k)$ such that
   for $G=G_{b,k}$, almost surely
   \[
      \reff^{G}\left(B_G(\rho, R) \leftrightarrow V(G) \setminus B_G(\rho, 2 R)\right) \leq C R^{\log_b(\Gamma_{b,k})},\quad
      \forall R \geq 1.
   \]
\end{theorem}
To see that this yields \pref{thm:intro-resist}, consider some $k \geq b^2$,
corresponding to the restriction $\mathsf{d}_g(b,k) \geq 3$.  Then for
all positive integers $p \geq 1$, we have
$\mathsf{d}_g(b^p,k^p)=\mathsf{d}_g(b,k)$ and
recalling \eqref{eq:dw-calc},
\[
   \lim_{p \to \infty} \log_{b^p}\!\left(\Gamma_{b^p,k^p}\right) = -1\,.
\]
To prove \pref{thm:annular-gbk}, it suffices to show the following.

\begin{lemma}\label{lem:annular-gn}
   For every $n \geq 1$, $k \geq b \geq 4$, there is a constant $C=C(k)$
   such that for $G=G(\bT^n_{(b,k)})$, we have
   \[
      \reff^{G}\left(B_G(x, R) \leftrightarrow V(G) \setminus B_G(x, 2 R)\right) \leq C R^{\log_b(\Gamma_{b,k})},\quad
      \forall x \in V(G), 1 \leq R \leq \diam(G)/C\,.
   \]
\end{lemma}

\begin{proof}
   Denote $\bT \seteq \bT_{(b,k)}$.
   Consider some value $1 \leq R \leq \diam(G)/C$, and
   define $h \seteq \lfloor \log_b(R/3)\rfloor$.

   Let $\cC_1,\ldots,\cC_{b^n}$ denote the columns of $\bT^n$ and
   writing $\bT^n = \bT^{n-h} \tileprod \bT^h$, let us partition the
   columns into consecutive sets $\cD_1,\ldots,\cD_{b^{n-h}}$ (as
   in the proof of \pref{lem:reff-times-conductance}), where
   $\cD_i = \cC_{(i-1) b^h + 1} \cup \cdots \cup \cC_{i b^h}$.
   For $1 \leq i \leq b^{n-h}$, let $\beta_i$ denote the number of tiles
   in the $i$th column of $\bT^{n-h}$ so that $\cD_i$ consists
   of $\beta_i$ copies of $\bT^h$ stacked vertically.

   Fix some vertex $x \in V(G)$ and
   suppose that $x \in \cD_s$ for some $1 \leq s \leq b^{n-h}$.
   Denote $\Delta \seteq 9b$.
   By choosing $C$ sufficiently large, we can assume that
   $b^{n-h} > 2\Delta$, so that either $s \leq b^{n-h} -\Delta$ or $s \geq 1+\Delta$.
   Let us assume that $s \leq b^{n-h}-\Delta$, as the other case is treated symmetrically.
   Define $t \seteq \lceil s+2+6b\rceil$ so that $t-s \leq \Delta$, and
   \begin{equation}\label{eq:st-dist}
      d_G(\cD_s, \cD_{t}) \geq b^h (t-s-1) \geq (t-s-1) \frac{R}{3b} > 2 R\,.
   \end{equation}

   Denote $\xi \seteq \gcd(\beta_s,\beta_{s+1},\ldots,\beta_t)$.
   We claim that
   \begin{equation}\label{eq:gcd}
      \xi \gtrsim_{k} \max(\beta_s,\beta_{s+1},\ldots,\beta_{t})\,.
   \end{equation}
   This follows because $\min(\beta_i,\beta_{i+1}) \mid \max(\beta_i,\beta_{i+1})$
   for all $1 \leq i < b^n$ (cf. \pref{obs:column}(b)),
   and moreover the ratio $\max(\beta_i,\beta_{i+1})/\min(\beta_i,\beta_{i+1})$
   is bounded by a function depending only on $k$.
   Since $t-s \lesssim_{k} 1$, this verifies \eqref{eq:gcd}.

   Denote $\hat{\cD} \seteq \cD_s \cup \cdots \cup \cD_t$.
   One can verify that $\hat{\cD}$ is a vertical stacking of $\xi$
   copies of $\hat{\bT} \seteq \bT_{\langle \beta_s/\xi,\ldots,\beta_t/\xi\rangle} \tileprod \bT^h$, and
   \pref{cor:side} implies that
   \begin{equation}\label{eq:hatT}
      \rho(\hat{\bT}) \lesssim_{k} \rho(\bT^h) \lesssim \Gamma_{b,k}^h\,,
   \end{equation}
   with the final inequality being the content of \pref{lem:eff-resist-general}. 

   Let $\hat{\bA}$ be the copy of $\hat{\bT}$ that contains $x$, and let
   $\bS$ be the copy of $\bT^h$ in $\bT^n = \bT^{n-h} \tileprod \bT^h$ that contains $x$.
   Since $\xi$ divides $\beta_s$, it holds that $\bS \subseteq \hat{\bA}$ and
   $\scrL(\bS) \subseteq \scrL(\hat{\bA})$.
   We further have
   \begin{equation*}\label{eq:rel-sizes}
      |\scrL(\hat{\bA})| = |\scrL(\hat{\bT})| = (\beta_s/\xi) |\scrL(\bT^h)| = (\beta_s/\xi) |\scrL(\bS)| \lesssim_{k} |\scrL(\bS)|\,.
   \end{equation*}
   This yields
   \begin{equation}\label{eq:eff-compare}
      \reff^G(\1_{\scrL(\bS)}/|\scrL(\bS)| \leftrightarrow \scrR(\hat{\bA}))
      \lesssim_{k}
      \reff^G(\1_{\scrL(\hat{\bA})}/|\scrL(\hat{\bA})| \leftrightarrow \scrR(\hat{\bA})),
   \end{equation}
   where we have used the hybrid notation:  For $s \in \ell_1(V)$,
   \[
      \reff^G(s \leftrightarrow U) \seteq \inf \left\{ \reff^G(s,t) : \supp(t) \subseteq U, \|t\|_1=\|s\|_1\right\}.
   \]

   Therefore,
   \begin{align*}
      \reff^G(\scrL(\bS) \leftrightarrow \scrR(\hat{\bA}))
      &\leq
      \reff^G(\1_{\scrL(\bS)}/|\scrL(\bS)| \leftrightarrow \scrR(\hat{\bA})) \\
      &\stackrel{\mathclap{\eqref{eq:eff-compare}}}{\lesssim_{k}}
      \reff^G(\1_{\scrL(\hat{\bA})}/|\scrL(\hat{\bA})| \leftrightarrow \scrR(\hat{\bA})) 
     =\rho(\hat{\bT})
      \stackrel{\eqref{eq:hatT}}{\lesssim_{k}} \Gamma_{b,k}^h\,.
   \end{align*}

   Since $\diam_G(\bS) \leq 3 b^h \leq R$ and $x \in \bS$, it holds that $\bS \subseteq B_G(x,R)$.
   On the other hand, since $x \in \bS$, \eqref{eq:st-dist} shows that $B_G(x, 2 R) \cap \scrR(\hat{\bA}) = \emptyset$.
   We conclude that
   \[
      \reff^G(B_G(x,R) \leftrightarrow V(G) \setminus B_G(x,2 R)) \leq
      \reff^G(\scrL(\bS) \leftrightarrow \scrR(\hat{\bA})) \lesssim_{k} \Gamma^h_{b,k} \lesssim_{k} R^{\log_b(\Gamma_{b,k})},
   \]
   as desired.
\end{proof}

\subsection{Complements of balls are connected}

Let us finally prove \pref{thm:intro-complement}.
Recall the setup from \pref{sec:hit-gbk}:  We take $k \geq b \geq 4$, define $G_n = G(\bT^n_{(b,k)})$, and use
$(\tilde{G}_n,\tilde{\rho}_n)$ to denote the cylindrical version,
which satisfies $\{(\tilde{G}_n,\tilde{\rho}_n)\} \todl (G_{b,k},\rho)$.

Partition the vertices of $\tilde{G}_n$ into columns $\cC_1, \ldots, \cC_{b^n}$ in
the natural way (as was done in \pref{sec:rw-speed} and \pref{sec:anular-resist}).
In what follows, we say that a set {\em $S \subseteq V(\tilde{G}_n)$ is connected} to
mean that $\tilde{G}_n[S]$ is a connected graph.
\begin{definition}
	Say that a set of vertices $U \subseteq V(\tilde{G}_n)$ is {\em vertically convex} 
	if 
	for all $1 \leq i \leq b^n$, either $U \cap \cC_i = \emptyset$ or 
   $U \cap \cC_i$ is connected.
\end{definition}

\begin{observation}
\label{obs:column-connected}
	Consider a connected set $U \subseteq \cC_i$ for some $1 \leq i \leq b^n$. Then for $1 \leq i' \leq b^n$ with $|i - i'| \leq 1$, 
   the set $B_{\tilde{G}_n}(U, 1) \cap \cC_{i'}$ is connected.
\end{observation}

Momentarily, we will argue that balls in $\tilde{G}_n$ are vertically convex.

\begin{lemma}
\label{lem:simply-connected}
   Consider any $A \in \bT$ and $R \geq 0$.
   It holds that $B_{\tilde{G}_n}(A,R)$ is vertically convex.
\end{lemma}

With this lemma in hand, it is easy to establish the following theorem
which, in conjunction with \pref{lem:gbk}, implies \pref{thm:intro-complement}.

\begin{theorem}
   Almost surely, the complement of every ball in $G_{b,k}$ is connected.
\end{theorem}

\begin{proof}
   Since $\{(\tilde{G}_n,\tilde{\rho}_n)\} \todl G_{b,k}$, it suffices to argue that
   for every $n \geq 1$, $x \in V(\tilde{G}_n)$, and $R \leq b^n/3$, the set $U \seteq V(\tilde{G}_n) \setminus B_{\tilde{G}_n}(x,R)$
      is connected.

      By \pref{lem:simply-connected}, it holds that $B_{\tilde{G}_n}(x,R)$ is vertically convex.
      Since the complement of a vertically convex set is vertically convex (given that every column in $\tilde{G}_n$
      is isomorphic to a cycle), $U$ is vertically convex as well.
      To argue that $U$ is connected, it therefore suffices to prove that there is a path from $\scrL(\bT^n_{(b,k)})$ to $\scrR(\bT^n_{(b,k)})$ in 
      $\tilde{G}_n[U]$.

      In fact, the tiles in $\bT^n_{(b,k)}$ have height at most $b^{-n}$, and therefore the projection of $K \seteq \region{B_{\tilde{G}_n}(x,R)}$ onto $\{0\} \times [0,1]$
      has length at most $(2R+1) b^{-n} < 1$.  Hence there is some height $h \in [0,1]$ such that a horizontal line $\ell$ at height $h$
      does not intersect $K$.  The set of tiles $\{ A \in V(\tilde{G}_n) : \ell \cap \region{A} \neq \emptyset \}$ therefore
      contains a path from $\scrL(\bT^n_{(b,k)})$ to $\scrR(\bT^n_{(b,k)})$ that is contained in $\tilde{G}_n[U]$, completing the proof.
\end{proof}

We are left to prove \pref{lem:simply-connected}.
To state the next lemma more cleanly, let us denote $\cC_0 = \cC_{b^n+1} = \emptyset$.

\begin{lemma}
\label{lem:column-simply-connected}
	Consider $1 \leq i \leq b^n$ and let $U \subseteq \cC_{i-1} \cup \cC_i \cup \cC_{i+1}$ be a connected, vertically convex subset of vertices.
   Then
   $B_{\tilde{G}_n}(U, 1) \cap \cC_i$
   is connected as well. 
\end{lemma}

\begin{proof}
	For $i' \in \{i-1, i, i+1\}$, denote $U_{i'} \seteq U \cap C_{i'}$. Clearly we have 
   \[B_{\tilde{G}_n}(U, 1) = B_{\tilde{G}_n}(U_{i-1}, 1) \cup B_{\tilde{G}_n}(U_i, 1) \cup B_{\tilde{G}_n}(U_{i+1}, 1)\,.\]
	 If $U_i = \emptyset$, then as $U$ is connected it should be that either $U = U_{i-1}$ or $U = U_{i+1}$, and in either case the claim follows by \pref{obs:column-connected}. 

    Now suppose $U_i \neq \emptyset$, and consider some $i' \in \{i-1, i+1\}$ such that $U_{i'} \neq \emptyset$.
	Then \pref{obs:column-connected} implies that $B_{\tilde{G}_n}(U_{i'}, 1) \cap \cC_i$ is connected.
   Furthermore, since $U$ is connected, 
   it holds that $B_{\tilde{G}_n}(U_{i'}, 1) \cap U_i \neq \emptyset$. Now, as we know 
   that $U_i$ is connected by the assumed vertical convexity of $U$,
   we obtain that $B_{\tilde{G}_n}(U, 1) \cap \cC_i$ is connected, completing the proof.
\end{proof}

\begin{proof}[Proof of \pref{lem:simply-connected}]
   We proceed by induction on $R$, where the base case $R=0$ is trivial, so
   consider $R \geq 1$. 
   Fix some $i \in \{1,2,\ldots,b^n\}$, and suppose $B_{\tilde{G}_n}(A, R) \cap \cC_i \neq \emptyset$. Denote \[U \seteq B_{\tilde{G}_n}(A, R - 1) \cap (\cC_{i-1} \cup \cC_i \cup \cC_{i+1})\,.\]
   Clearly we have $B_{\tilde{G}_n}(A, R) \cap \cC_i = B_{\tilde{G}_n}(U, 1) \cap \cC_i$.
   The set $U$ is manifestly connected and, by the induction hypothesis, is also vertically convex.
   Thus from \pref{lem:column-simply-connected}, it follows that $B_{\tilde{G}_n}(U, 1) \cap \cC_i$ is connected as well. 
   We conclude that $B_{\tilde{G}_n}(A, R)$ is vertically convex, completing the proof.
\end{proof}

\subsection*{Acknowledgements}

We thank Shayan Oveis Gharan and Austin Stromme for many useful preliminary
discussions, Omer Angel and Asaf Nachmias
for sharing with us their construction of a graph with asymptotic $(3-\e)$-dimensional volume growth
on which the random walk has diffusive speed,
and Itai Benjamini for emphasizing many of the questions addressed here.
We also thank the anonymous referees for very useful comments.
This research was partially supported by NSF CCF-1616297 and a Simons Investigator Award.

\bibliographystyle{alpha}
\bibliography{diffusive}

\newcommand{\etalchar}[1]{$^{#1}$}
\begin{thebibliography}{CRR{\etalchar{+}}97}

\bibitem[ADJ97]{adj97}
Jan Ambj{\o}rn, Bergfinnur Durhuus, and Thordur Jonsson.
\newblock {\em Quantum geometry}.
\newblock Cambridge Monographs on Mathematical Physics. Cambridge University
  Press, Cambridge, 1997.
\newblock A statistical field theory approach.

\bibitem[AL07]{aldous-lyons}
David Aldous and Russell Lyons.
\newblock Processes on unimodular random networks.
\newblock {\em Electron. J. Probab.}, 12:no. 54, 1454--1508, 2007.

\bibitem[Ang03]{angel03}
O.~Angel.
\newblock Growth and percolation on the uniform infinite planar triangulation.
\newblock {\em Geom. Funct. Anal.}, 13(5):935--974, 2003.

\bibitem[Bar98]{Barlow98}
Martin~T. Barlow.
\newblock Diffusions on fractals.
\newblock In {\em Lectures on probability theory and statistics
  ({S}aint-{F}lour, 1995)}, volume 1690 of {\em Lecture Notes in Math.}, pages
  1--121. Springer, Berlin, 1998.

\bibitem[BC12]{bc12}
Itai Benjamini and Nicolas Curien.
\newblock Ergodic theory on stationary random graphs.
\newblock {\em Electron. J. Probab.}, 17:no. 93, 20 pp., 2012.

\bibitem[BC13]{bc13}
Itai Benjamini and Nicolas Curien.
\newblock Simple random walk on the uniform infinite planar quadrangulation:
  subdiffusivity via pioneer points.
\newblock {\em Geom. Funct. Anal.}, 23(2):501--531, 2013.

\bibitem[Ben13]{BenStFlour13}
Itai Benjamini.
\newblock {\em Coarse geometry and randomness}, volume 2100 of {\em Lecture
  Notes in Mathematics}.
\newblock Springer, Cham, 2013.
\newblock Lecture notes from the 41st Probability Summer School held in
  Saint-Flour, 2011, Chapter 5 is due to Nicolas Curien, Chapter 12 was written
  by Ariel Yadin, and Chapter 13 is joint work with Gady Kozma, \'{E}cole
  d'\'{E}t\'{e} de Probabilit\'{e}s de Saint-Flour. [Saint-Flour Probability
  Summer School].

\bibitem[BK02]{bk02}
Mario Bonk and Bruce Kleiner.
\newblock Quasisymmetric parametrizations of two-dimensional metric spheres.
\newblock {\em Invent. Math.}, 150(1):127--183, 2002.

\bibitem[BP11]{BP11}
Itai Benjamini and Panos Papasoglu.
\newblock Growth and isoperimetric profile of planar graphs.
\newblock {\em Proc. Amer. Math. Soc.}, 139(11):4105--4111, 2011.

\bibitem[BS01]{bs01}
Itai Benjamini and Oded Schramm.
\newblock Recurrence of distributional limits of finite planar graphs.
\newblock {\em Electron. J. Probab.}, 6:no. 23, 13 pp., 2001.

\bibitem[CRR{\etalchar{+}}97]{CRRST96}
Ashok~K. Chandra, Prabhakar Raghavan, Walter~L. Ruzzo, Roman Smolensky, and
  Prasoon Tiwari.
\newblock The electrical resistance of a graph captures its commute and cover
  times.
\newblock {\em Comput. Complexity}, 6(4):312--340, 1996/97.

\bibitem[GN13]{gn13}
Ori {Gurel-Gurevich} and Asaf Nachmias.
\newblock Recurrence of planar graph limits.
\newblock {\em Ann. of Math. (2)}, 177(2):761--781, 2013.

\bibitem[KM08]{KM08}
Takashi Kumagai and Jun Misumi.
\newblock Heat kernel estimates for strongly recurrent random walk on random
  media.
\newblock {\em J. Theoret. Probab.}, 21(4):910--935, 2008.

\bibitem[Lee17]{lee17a}
James~R. Lee.
\newblock Conformal growth rates and spectral geometry on distributional limits
  of graphs.
\newblock Preprint at
  \href{https://arxiv.org/abs/1701.01598}{arXiv:math/1701.01598}, 2017.

\bibitem[LP16]{lp:book}
Russell Lyons and Yuval Peres.
\newblock {\em Probability on Trees and Networks}.
\newblock Cambridge University Press, New York, 2016.
\newblock Preprint at \url{http://pages.iu.edu/~rdlyons/}.

\bibitem[Mur19]{Murugan19}
Mathav Murugan.
\newblock Quasisymmetric uniformization and heat kernel estimates.
\newblock {\em Trans. Amer. Math. Soc.}, 372(6):4177--4209, 2019.

\end{thebibliography}

\end{document}